\documentclass[11pt,a4paper]{article}
\usepackage{amsmath, amsfonts, amssymb,amsthm}
\usepackage[left=1in, right=1in,top=1in,bottom=1in]{geometry}
\usepackage{enumitem}
\usepackage{xcolor}
\usepackage{pstricks}
\usepackage{hyperref}
\usepackage{latexsym}
\usepackage{cite}
\DeclareMathOperator*{\esssup}{ess\,sup}
\DeclareMathOperator*{\essinf}{ess\,inf}
\allowdisplaybreaks

\def\R{{\mathbb R}}

\numberwithin{equation}{section}
\usepackage{comment}
\newtheorem{definition}{Definition}[section]
\newtheorem{theorem}{Theorem}[section]

\newtheorem{lemma}{Lemma}[section]
\newtheorem{prop}{Proposition}[section]

\def\dxt{\,{\rm d}x{\rm d}t}

\def\dx{\,{\rm d}x}
\def\dtau{\,{\rm d}\tau}
\numberwithin{equation}{section}
\begin{document}
\setlength{\abovedisplayskip}{3pt}
\setlength{\belowdisplayskip}{3pt}
\date{}
\title{On some Elliptic and Parabolic Problems Involving the Anisotropic $\vec{\textbf{p}}(u)$-Laplacian}
\author{ {\bf Kaushik Bal$\,^{1,}$\footnote{e-mail: {\tt kaushik@iitk.ac.in}},  Shilpa Gupta$\,^{1,}$\footnote{e-mail: {\tt shilpagupta890@gmail.com}}} \\ 
        $^1\,$Department of Mathematics and Statistics,\\ Indian Institute of Technology Kanpur,\\Uttar Pradesh, 208016, India}
\maketitle
\begin{abstract}
We investigate a class of elliptic and parabolic partial differential equations characterized by anisotropic \(\vec{\textbf{p}}(u)\)-Laplace operator, where the vector-valued exponent \(\vec{\textbf{p}} = (p_1, \ldots, p_N)\) depends on the unknown function \(u\) and a non-local function of $u$, respectively. This dependence necessitates the use of variable exponent Sobolev spaces specifically tailored to the anisotropic framework. For the elliptic case, we establish the existence of a weak solution by employing the theory of pseudomonotone operators in conjunction with suitable approximation techniques. In the parabolic setting, the existence of a weak solution is obtained via a time discretization scheme and Schauder’s fixed-point theorem, supported by a priori estimates and compactness arguments.


\medskip

\end{abstract}

%\tableofcontents
\section{Introduction}\setcounter{equation}{0}
\indent This paper is focused on establishing the existence of weak solutions for a class of elliptic and parabolic partial differential equations that involve anisotropic $p(u)$-Laplace operators. These operators are characterized by a vector-valued exponent that depends on the unknown function $u$ in the elliptic case and on a nonlocal function of $u$ in the parabolic case. We begin with the following elliptic problem:
\begin{equation}\label{1.1}
	-\Delta_{\vec{\textbf{p}}(u)}u  =f(x,u)\quad \text{in } \Omega;\ \quad u = 0 \quad \text{on}\; \partial\Omega,
	\end{equation}
where $\Omega\subset\R^{N}(N\geq 2)$ is a bounded domain with Lipschitz boundary $\partial\Omega$,
$$-\Delta_{\vec{\textbf{p}}(u)} u:=\sum_{i=1}^{N}\frac{\partial}{\partial x_i}\left( \left|\frac{\partial u}{\partial  x_i} \right|^{p_i(u)-2}\frac{\partial u}{\partial  x_i} \right), $$
where the exponent vector $\vec{\textbf{p}}:=(p_1,p_2,\ldots,p_N)$ with $p_i:\R\rightarrow [2,\infty)$ being continuous for each $i=1,2,\ldots,N$. The nonlinear function $f:\Omega\times\mathbb{R}\to\mathbb{R}$ is assumed to satisfy a set of suitable conditions, which will be specified later.

We further consider the corresponding nonlocal parabolic problem:
\begin{equation}\label{p1.1}
\begin{cases}
u_t - \Delta_{\vec{\mathbf{p}}(b(u))}u = f, & \text{in } \Omega \times (0, T), \\
u = 0, & \text{on } \Gamma := \partial\Omega \times (0, T), \\[4pt]
u(x,0) = u_0(x), & \text{in } \Omega.
\end{cases}
\end{equation}

where $\Omega\subset\R^{N}(N\geq 2)$ is a bounded domain with Lipschitz boundary $\partial\Omega$,  $f\in W^{-1,(p^{-})'}(\Omega)$, the initial datum $u_0$ belongs to $L^2(\Omega)$, and $b: W^{1,p^-}{0}(\Omega)\to \mathbb{R}$ is a continuous and bounded mapping, where the notation $p^-$ will be clarified later. The prototypical example of such nonlocal mappings include $$b(u)=\|\nabla u\|_{L^{p^{-}}(\Omega)}  \ \ \text{ and } \ \ b(u)=\|u\|_{L^{q}(\Omega)} \text{ for } q\leq (p^{-})^{*}=\frac{N p^{-}}{N-p^{-}}.$$

Due to the presence of an unknown function in the exponent, the main difficulty is that the problems \eqref{1.1} and \eqref{p1.1} can not be written as an equality in terms of duality pairing in a fixed Banach space. In fact, two distinct solutions may naturally belong to two different Sobolev spaces, depending on the corresponding values of the exponent. To the best of our knowledge, the first systematic study of a $p(u)$-Laplacian problem was carried out by Andreianov–Bendahmane–Ouaro \cite{andreianov2010structural}. They considered the elliptic boundary value problem
\begin{equation*}
\begin{cases}
u - \Delta_{p(u)}u = f, & \text{in } \Omega, \\
u = 0, & \text{on } \partial\Omega.
\end{cases}
\end{equation*}
under the suitable regularity assumptions on the domain $\Omega$. By exploiting techniques that effectively reduce the analysis to the setting of the Lebesgue space $L^1$ they established the existence of broad and narrow weak solutions. Subsequently, Chipot–Oliveira \cite{chipot2019some} proposed a different approach and studied both local and nonlocal formulations of the $p(u)$-Laplacian problem, namely
\begin{equation}\label{1_3}
\begin{cases}
-\Delta_{p(u)}u = f, & \text{in } \Omega, \\
u = 0, & \text{on } \partial\Omega.
\end{cases}
\end{equation}
and 
\begin{equation*}
\begin{cases}
-\Delta_{p(b(u))}u = f, & \text{in } \Omega, \\
u = 0, & \text{on } \partial\Omega.
\end{cases}
\end{equation*}
Their analysis was based on the Minty monotonicity trick combined with the powerful techniques introduced by Zhikov \cite{zhikov2009technique}, which are particularly well-suited for handling problems with nonstandard growth conditions. In the variable exponent context, these methods require additional care due to the non-homogeneity of the norm. In the same work, Chipot–Oliveira also formulated a collection of open problems, some of which have been addressed in later research. In particular, Zhang–Zhang \cite{zhang2021some} partially solved these questions by proving the existence of entropy solutions to the local elliptic problem \eqref{1_3}. Furthermore, they analyzed the parabolic extension
\begin{equation*}
\begin{cases}
u_t - \Delta_{p(b(u))}u = f, & \text{in } \Omega \times (0, T), \\
u = 0, & \text{on } \Gamma := \partial\Omega \times (0, T), \\
u(x,0) = u_0(x), & \text{in } \Omega.
\end{cases}
\end{equation*}

where the interplay between the nonlinear diffusion and the time evolution requires additional compactness and regularity tools. Further progress in the study of parabolic problems involving \(p[u(x,t)]\)-Laplacian operators was made by Antontsev-Shmarev \cite{antontsev2020class}, who analyzed such problems under the assumption that the co-domain of the exponent function lies within the interval \((1,2)\). This analysis was later extended by Antontsev-Kuznetsov-Shmarev~\cite{antontsev2021class}, where the dependence on \(u(x,t)\) was replaced by a dependence on the gradient \(\nabla u\), leading to the study of nonlocal parabolic problems governed by the \(p[\nabla u]\)-Laplacian. Despite these developments, parabolic problems involving the \(p(u)\)-Laplacian operator have received comparatively less attention in the literature. Notably, Aouaoui-Bahrouni \cite{aouaoui2022some}, as well as Aouaoui \cite{aouaoui2023existence}, established existence results for \(p(u)\)-Laplacian type equations posed in the whole space \(\mathbb{R}^N\). In the anisotropic setting, Giacomoni-Vallet \cite{giacomoni2012some} studied parabolic problems involving the \(p(x)\)-Laplacian operator.  More recently, Bahrouni-Bahrouni-Missaoui \cite{bahrouni2024new} investigated double-phase equations with exponents depending on the gradient of the solution, broadening the class of variable exponent problems. 

Motivated by the above work, we address the problems \eqref{1.1} and \eqref{p1.1} in this paper. Anisotropic $p(u)$-Laplacian problems capture complex behaviors arising in media with direction-dependent properties and modeling phenomena that cannot be addressed by isotropic equations alone. Our analysis combines the anisotropic, variable exponent framework with suitable approximation techniques, enabling us to extend the theory of $p(u)$-Laplacian type operators in new directions. 

We assume that non-linear function $f:\Omega\times\R\rightarrow \R$ is a Carath\'eodory function such that $f(\cdot,0)<0$ and fulfills the following condition:
\begin{itemize}
\item[$(f)$] $|f(x,t)|\leq c(1+|t|^{r-1}),$  $\forall \ (x,t)\in \Omega\times\R,$ for some $1\leq r< p^{-}$ and $c>0$.
\end{itemize}

We assume that $\vec{\textbf{p}}=(p_1,p_2,\ldots,p_N)$, $p_i: \mathbb{R} \rightarrow [2,\infty)$ for all $i=1,2,\ldots,N$ are continuous functions that fulfill the following conditions:
\begin{itemize}
\item[$ (p_{1}) $]  $N<p_i^-:=\essinf\limits_{x\in\R}p_i(t)\leq p_i(t)\leq\esssup\limits_{t\in\R}p_i(t):=p_i^+<\infty,\forall \  t\in \R$ and $i=1,2,\ldots,N$.
%\item[$ (p'_{1}) $]  $N<p_i^-:=\essinf\limits_{x\in\R}p_i(t)\leq p_i(t)\leq\esssup\limits_{t\in\R}p_i(t):=p_i^+<\infty,\forall \  t\in \R$ and $i=1,2,\ldots,N$. 
\item[$ (p_{2}) $] For each $i\in\{1,2,\ldots,N\}$ the exponent function $p_i$  is Lipschitz continuous, i.e.,  there exist $c_i>0$ such that
$$|p_i(t_1)-p_i(t_2)|\leq c_i |t_1-t_2| , \text{ for } t_1,t_2\in \R.$$
%\item[$ (p_{3}) $] For all $t\in \R$, we have $\overline{p}(t)> N$, where 
%$\overline{p}(t)=\dfrac{N}{\sum\limits_{i=1}^{N}\frac{1}{p_i(t)}}$.
\end{itemize}

We now state the main results of this article.
\begin{theorem}\label{t1}
Assume that conditions $(f)$ and $ (p_{1})$-$(p_{2})$ hold. Then problem \eqref{1.1} admits a non-trivial weak solution.
\end{theorem}
\begin{theorem}\label{t2}
Suppose that condition $(p_{1})$ is satisfied, $f \in W^{-1,(p^{-})'}(\Omega)$ and $u_0\in L^2(\Omega)$. Let $b: W^{1,p^-}_{0}(\Omega) \to \mathbb{R}$ be a continuous and bounded mapping. Then problem \eqref{p1.1} admits a non-trivial weak solution in the sense of  definition \ref{para_weak_sol_def}.
\end{theorem}
To address the challenges posed by the solution-dependent exponent in the operator, we employ an approximation technique inspired by  Chipot-Oliveira \cite{chipot2019some}. 

Specifically, to prove Theorem~\ref{t1}, we followed the following technique:
\begin{itemize}
\item We begin by formulating  a perturbed version of the original problem \eqref{1.1}, given by \eqref{1.2}, where a regularizing term involving the $p^+$-Laplacian is added, multiplied by a small parameter $\epsilon > 0$.
\item The inclusion of this perturbed term ensures that the operator is dominated by the higher order regularizing term, leveraging the fact that $\esssup\limits_{t\in\mathbb{R}} p_i(t) =: p_i^+$. This dominance allows us to obtain uniform a priori estimates, which are crucial in the existence analysis.
\item Using these estimates and the theory of pseudomonotone operators (Theorem~\ref{the1}), we prove the existence of a weak solution to the perturbed problem~\eqref{1.2} in Theorem~\ref{per_theorem}.
\item Finally, in Subsection~\ref{subsec3.2}, we pass to the limit as $\epsilon \to 0$ and rigorously justify the convergence of the approximating sequence, thereby obtaining the existence of a weak solution to the original problem \eqref{1.1}.
\end{itemize}

To prove Theorem~\ref{t2}, we proceed through the following sequence of well-structured steps:
\begin{itemize}
\item We begin by partitioning the time interval 
$(0,T)$ into $N_0$ subintervals of uniform length $h=T/N_0$. For each discrete time level, we consider the corresponding time-discrete problem~\eqref{elli_pro_k}, which takes the form of an elliptic equation.
\item Due to the dependence of the unknown function in the exponent of the operator, we introduce a perturbed version of the problem~\eqref{elli_pro_k}  given by \eqref{apprx_pro}, where a regularizing term involving the $p^+$-Laplacian is added, multiplied by a small parameter $\epsilon > 0$.  
\item As a preliminary step, we fix the exponent in problem \eqref{apprx_pro} and analyze the modified problem~\eqref{apprx_pro_0}. By applying the theory of monotone operators, we observe that there exists a unique solution to this simplified problem.
\item Employing Schauder’s fixed point theorem, we demonstrate the existence of a weak solution to the perturbed problem \eqref{apprx_pro}. 
\item With uniform estimates in hand, we pass to the limit as $\epsilon \to 0$ to recover a weak solution of the original time-discrete elliptic problem~\eqref{elli_pro_k}. 
\item Finally, to recover a weak solution to the original parabolic problem~\eqref{p1.1}, we let the time-step size $h\to 0$ and perform a careful convergence analysis. This yields the desired existence result for the full time-dependent problem.
\end{itemize}
\textbf{Notation:} Throughout the paper we adopt the following conventions:
\begin{enumerate}
\item[(i)] $\Omega_T:=\Omega\times(0,T)$.
\item[(ii)] $C$ denotes a generic positive constant, whose value may vary from line to line.
\item[(iii)] For $k \in (1, \infty)$, $k' := \frac{k}{k-1}$ is the conjugate exponent of $k$. 
\item[(iv)] $C_{+}(\overline{\Omega})=\{q\in C(\overline{\Omega},\R):\inf\limits_{x\in \Omega}q(x)>1\}.$
\item[(v)] $q^-:=\inf\limits_{x\in \Omega}q(x)$ and $q^+:=\sup\limits_{x\in \Omega}q(x)$.
\item[(vi)] Without loss of generality, we assume $p^-:=p_1^-\leq p_2^-\leq \ldots \leq p_N^-\leq p_1^+ \leq p_2^+ \ldots \leq p_N^+:=p^+$. 
\end{enumerate}

The paper is organized as follows: Section \ref{sec2} focusing on  the suitable Sobolev spaces, which are essential for handling the non-standard operator $\vec{\textbf{p}}(u)$-Laplacian. Section \ref{sec3} is devoted to the analysis of the nonlinear elliptic problem \eqref{1.1}, where we establish the existence of weak solutions using the theory of pseudomonotone operators together with perturbation method. Finally, in Section \ref{sec4}, we consider the associated parabolic problem \eqref{p1.1} and prove the existence of weak solutions by employing a combination of time discretization, approximation arguments, and Schauder’s fixed point theorem.

\section{Functional spaces and auxiliary results}\label{sec2}
Examining the elliptic equation~\eqref{1.1}, we observe that the exponent vector $\vec{\textbf{p}}$ depends on the solution~$u$, which itself is determined by the space variable $x$. Consequently, for a given function $u$, the exponent can ultimately be written as a function of \( x \) in the form of a variable exponent \( \vec{\textbf{q}}(x) \), where $\vec{\textbf{q}}(x) = \vec{\textbf{p}}(u(x)).$ Therefore, the natural space to study the equation \eqref{1.1} is the anisotropic variable exponent Sobolev space. 
 In contrast, for the parabolic equation~\eqref{p1.1}, the exponent~$\vec{\textbf{p}}$ depends on the function~$b$, which in turn is determined by~$u$. As a result, for a given~$u$, the exponent $\vec{\textbf{p}}(b(u))$ is a vector in~$\mathbb{R}^N$. Hence, the suitable space for analyzing the equation~\eqref{p1.1} is the anisotropic Sobolev space.

\subsection{Variable exponent Lebesgue spaces}
Let $\Omega\subset\R^{N}$ be a bounded domain with Lipschitz boundary $\partial\Omega$. For $q\in C_{+}(\overline{\Omega})$, variable exponent Lebesgue space \( L^{q(\cdot)}(\Omega) \) is defined by
\[
L^{q(\cdot)}(\Omega) := \left\{ u : \Omega \to \mathbb{R} \text{ measurable} \ \bigg| \ \int_\Omega |u(x)|^{q(x)} \dx< \infty \right\}
\]
which is a norm space with the luxemburg norm 
$$\|u\|_{L^{q(\cdot)}(\Omega)}=\inf\left\lbrace \tau>0: \ \int_{\Omega}\left| \frac{u(x)}{\tau} \right|^{q(x)}   \dx\leq 1  \right\rbrace \cdot$$ 
The space $L^{q(\cdot)}(\Omega)$ is Banach, reflexive and separable \cite{kovavcik}.

\begin{prop}\cite[Theorem 2.1]{kovavcik}
Let $r\in C_{+}(\overline{\Omega})$ and $s\in C_{+}(\overline{\Omega})$ be the conjugate exponents, i.e., $1/r(x)+1/s(x)=1$  $\forall x\in \overline{\Omega}.$ Then, for any $u\in L^{r(\cdot)}(\Omega)$ and $v\in L^{s(\cdot)}(\Omega)$, we have
$$\left| \int_{\Omega}uv \dx\right|\leq\left(\frac{1}{r^{-}}+\frac{1}{s^{-}}\right) \|u\|_{L^{r(\cdot)}(\Omega)}\|v\|_{L^{s(\cdot)}(\Omega)}.$$
\end{prop}
\begin{prop}\cite{fan2001spaces}\label{rel}
Let $q\in C_{+}(\overline{\Omega})$. For any $u\in L^{q(\cdot)}(\Omega),$ the followings are true:
\begin{enumerate}
\item   $\|u\|_{L^{q(\cdot)}(\Omega)}^{q^{-}}\leq \rho(u)\leq\|u\|_{L^{q(\cdot)}(\Omega)}^{q^{+}}$ whenever ${\|u\|_{L^{q(\cdot)}(\Omega)}}> 1,$
\item   $\|u\|_{L^{q(\cdot)}(\Omega)}^{q^{+}}\leq \rho(u)\leq\|u\|_{L^{q(\cdot)}(\Omega)}^{q^{-}}$ whenever ${\|u\|_{L^{q(\cdot)}(\Omega)}}<1,$
\item  ${\|u\|_{L^{q(\cdot)}(\Omega)}}<1(=1;>1)$ iff  $\rho(u)<1(=1;>1)$,
\end{enumerate}
where $\rho(u)=\int_\Omega |u|^{q(x)}\dx.$
\end{prop}
To know more about these spaces, one can check \cite{fan2001spaces,kovavcik,radulescu2015partial,cruz2013variable,barilla2021existence}.
\subsection{Anisotropic variable exponent Sobolev spaces} Let $\vec{\textbf{p}}=(p_1,p_2,\ldots,p_N)$, where $p_i\in C_{+}(\overline{\Omega})$ for all $i=1,2,\ldots,N$. For $x\in\Omega$, we define
$$p_{M}(x)=\max\{p_1(x),p_2(x),\ldots,p_N(x)\},$$
$$\overline{p}(x)=\dfrac{N}{\sum_{i=1}^{N}\frac{1}{p_i(x)}}$$
 and 
 \[
\overline{p}^*(x) =
\begin{cases}
\frac{N\,\overline{p}(x)}{N - \overline{p}(x)}, & \text{if } \overline{p}(x) < N, \\[6pt]
+\infty, & \text{if } \overline{p}(x) \ge N.
\end{cases}
\]

We introduce the anisotropic variable exponent Sobolev space $W^{1,\vec{\textbf{p}}(\cdot)}(\Omega)$ as 
\begin{align*}
W^{1,\vec{\textbf{p}}(\cdot)}(\Omega)&=\left\lbrace v\in L^{p_M(\cdot)}(\Omega) :  \frac{\partial v}{\partial  x_i}\in L^{p_i(\cdot)}(\Omega), \ \forall \ i=1,2,\ldots,N \right\rbrace \\
&=\left\lbrace v\in L^{1}_{loc}(\Omega) : v\in L^{p_i(\cdot)}(\Omega), \  \frac{\partial v}{\partial  x_i}\in L^{p_i(\cdot)}(\Omega), \ \forall \ i=1,2,\ldots,N \right\rbrace
\end{align*}
which is a norm space with the norm
$$\|v\|_{W^{1,\vec{\textbf{p}}(\cdot)}(\Omega)}=\|v\|_{L^{p_M(\cdot)}(\Omega)}+\sum_{i=1}^{N}\left\| \frac{\partial v}{\partial  x_i}\right\|_{ L^{p_i(\cdot)}(\Omega)} $$
where
$$\left\|  \frac{\partial v}{\partial  x_i}\right\| _{L^{p_i(\cdot)}(\Omega)}=\inf\left\lbrace \tau>0: \ \int_{\Omega}\left| \frac{\partial v}{\tau\partial  x_i} \right|^{p_i(x)}   \dx\leq 1  \right\rbrace \cdot$$ 
Next, consider the closure of $C_c^\infty(\Omega)$ with respect to the space  $W^{1,\vec{\textbf{p}}(\cdot)}(\Omega)$ and denote it as $W^{1,\vec{\textbf{p}}(\cdot)}_{0}(\Omega)$, i.e., $$W^{1,\vec{\textbf{p}}(\cdot)}_{0}(\Omega)=\overline{C_c^\infty(\Omega)}|^{\|\cdot\|_{W^{1,\vec{\textbf{p}}(\cdot)}(\Omega)}}.$$
Also, define the space
\[
\mathring{W}^{1,\vec{\textbf{p}}(\cdot)}(\Omega) 
= \left\{ u \in W^{1,\vec{\textbf{p}}(\cdot)}(\Omega) \; : \; u|_{\partial \Omega} = 0 \right\}
\]
Let $\Omega$ be a bounded domain with Lipschitz boundary $\partial\Omega$. Then, by definition,
\[
\mathring{W}^{1,\vec{\textbf{p}}(\cdot)}(\Omega) = W^{1,1}_0(\Omega) \cap W^{1,\vec{\textbf{p}}(\cdot)}(\Omega),
\]
and it is clear that
$W^{1,\vec{\textbf{p}}(\cdot)}_0(\Omega) \subset \mathring{W}^{1,\vec{\textbf{p}}(\cdot)}(\Omega).$

In the constant exponent setting, i.e., when $\vec{\textbf{p}} = (p_1, p_2, \dots, p_N) \in \mathbb{R}^N$, these spaces coincide:
\[
W^{1,\vec{\textbf{p}}}_0(\Omega) = \mathring{W}^{1,\vec{\textbf{p}}}(\Omega).
\]

However, for variable exponents, this equality generally fails, that is,
\[
W^{1,\vec{\textbf{p}}(\cdot)}_0(\Omega) \ne \mathring{W}^{1,\vec{\textbf{p}}(\cdot)}(\Omega),
\]
and the space of smooth compactly supported functions $C_c^\infty(\overline{\Omega})$ is not necessarily dense in $W^{1,\vec{\textbf{p}}(\cdot)}(\Omega)$.

Concerning the density of smooth functions in $W^{1,\vec{\textbf{p}}(\cdot)}_0(\Omega)$, the following result holds; to establish this density, some additional assumptions on the variable exponents are required.

\begin{theorem}\cite[Theorem 2.4]{fan2011anisotropic}\label{denseness}(Denseness)
Let $\Omega\subset \R^N$ be a bounded domain with a Lipschitz boundary and $\vec{\textbf{p}}=(p_1,p_2,\ldots,p_N)\in(C_{+}(\overline{\Omega}))^{N}$. Assume that, for each $i=1,2,\ldots,N$, $p_i$ is log H\"older continuous, i.e.,  
$$|p_i(x_1)-p_i(x_2)|\leq \dfrac{L}{\ln\left(\dfrac{1}{|x_1-x_2|}\right)}, \text{ for } x_1,x_2\in \overline{\Omega} \text{ whenever } |x_1-x_1|\leq \dfrac{1}{2}.$$ Then $C_c^\infty(\Omega)$ is dense in $\mathring{W}^{1,\vec{\textbf{p}}(\cdot)}(\Omega)$. Moreover, $W^{1,\vec{\textbf{p}}(\cdot)}_{0}(\Omega)= \mathring{W}^{1,\vec{\textbf{p}}(\cdot)}(\Omega)$.
\end{theorem}
\begin{theorem}\label{denseness_0}\cite[Theorem 2.5]{fan2011anisotropic}(Regularity)
Let $\Omega\subset \R^N$ be a bounded domain with a Lipschitz boundary and $\vec{\textbf{p}}=(p_1,p_2,\ldots,p_N)\in(C_{+}(\overline{\Omega}))^{N}$. Assume that, $\overline{p}(x)> N$ for all $x\in\Omega$.  Then there exists $\alpha\in(0,1)$ such that $W^{1,\vec{\textbf{p}}(\cdot)}_{0}(\Omega)$ is continuously embedded in $C
^{0,\alpha}(\overline{\Omega})$.
\end{theorem}
\begin{theorem}\cite[Theorem 2.6]{fan2011anisotropic}(Poincar\'e inequality)
Let $\Omega\subset \R^N$ be a bounded domain with a Lipschitz boundary and $\vec{\textbf{p}}=(p_1,p_2,\ldots,p_N)\in(C_{+}(\overline{\Omega}))^{N}$. Assume that, $p_M(x)\leq \overline{p}^*(x)$ for all $x\in\Omega$. Then, we have the following Poincar\'e inequality:
$$\|v\|_{L^{p_M(\cdot)}(\Omega)}\leq c\sum_{i=1}^{N}\left\| \frac{\partial v}{\partial  x_i}\right\|_{ L^{p_i(\cdot)}(\Omega)}, \forall \ v\in W^{1,\vec{\textbf{p}}(\cdot)}_{0}(\Omega)$$
for some $c>0$. Thus, $\sum_{i=1}^{N}\left\| \frac{\partial v}{\partial  x_i}\right\|_{ L^{p_i(\cdot)}(\Omega)}$ is an equivalent norm in $W^{1,\vec{\textbf{p}}(\cdot)}_{0}(\Omega)$.
\end{theorem}

If the exponents are constant, i.e., $\vec{\textbf{p}}=(p_1,p_2,\ldots,p_N)\in\R^{N}$  then the space reduces to the anisotropic Sobolev space
$$W^{1,\vec{\textbf{p}}}(\Omega)=\left\lbrace v\in L^{p_M}(\Omega) :  \frac{\partial v}{\partial  x_i}\in L^{p_i}(\Omega), \ \forall \ i=1,2,\ldots,N \right\rbrace$$
which is a norm space with the norm
$$\|v\|_{W^{1,\vec{\textbf{p}}}(\Omega)}=\|v\|_{L^{p_M}(\Omega)}+\sum_{i=1}^{N}\left\| \frac{\partial v}{\partial  x_i}\right\|_{ L^{p_i}(\Omega)}$$
where $p_M=\max\{p_1,p_2,\ldots,p_N\}$. Next, consider the closure of $C_c^\infty(\Omega)$ with respect to the space  $W^{1,\vec{\textbf{p}}}(\Omega)$ and denote it as $W^{1,\vec{\textbf{p}}}_{0}(\Omega)$, i.e., $$W^{1,\vec{\textbf{p}}}_{0}(\Omega)=\overline{C_c^\infty(\Omega)}|^{\|\cdot\|_{W^{1,\vec{\textbf{p}}}(\Omega)}}$$ which is a norm space with the norm
$$\|v\|_{W_0^{1,\vec{\textbf{p}}}(\Omega)}=\sum_{i=1}^{N}\left\| \frac{\partial v}{\partial  x_i}\right\|_{L^{p_i}(\Omega)}.$$
For further details on anisotropic Sobolev spaces, we refer to\cite{chrif2024renormalized,razani2024positive,bendahmane2011approximation}.
\subsection{Auxiliary results}
 To prove our main result, we will use the following inequalities:
\begin{lemma}\cite[Lemma 2.1]{misawa2023finite}\label{inequalities}
For every $p\in(1,\infty)$ there exists $C_1,C_2,C_3>0$ such that for all $x,y\in \R^{N},$ where $ \langle\cdot,\cdot \rangle $ is the usual inner product in $\R^{N}$, the following inequalities holds:
$$
\langle|x|^{p-2}x-|y|^{p-2}y,x-y\rangle\geq C_{1}(|x|+|y|)^{p-2}|x-y|^{2},
$$
and
$$
||x|^{p-2}x-|y|^{p-2}y|\leq C_{2}(|x|+|y|)^{p-2}|x-y|.
$$
In particular, for $p\geq 2$
$$
\langle|x|^{p-2}x-|y|^{p-2}y,x-y\rangle\geq C_{3}|x-y|^{p}.
$$
\end{lemma}
\begin{definition}
Let $X$ be a reflexive Banach space, and let $\langle \cdot, \cdot \rangle_{X}$ denote the duality pairing between $X$ and its dual space $X^{*}$. Let $J:X \to X^{*}$ be an operator. We say:

\begin{enumerate}
    \item $J$ is \emph{bounded} if it maps bounded sets in $X$ into bounded sets in $X^{*}$.
    
    \item $J$ is \emph{coercive} if
    \[
    \lim_{\|u\|\to \infty} \frac{\langle J(u), u \rangle_{X}}{\|u\|} = \infty.
    \]
    
    \item $J$ is \emph{pseudomonotone} if whenever $\{u_n\} \subset X$ converges weakly to $u$ in $X$ and
    \[
    \limsup_{n \to \infty} \langle J(u_n), u_n - u \rangle_{X} \le 0,
    \]
    then $J(u_n) \rightharpoonup J(u)$ in $X^{*}$ and
    \[
    \langle J(u_n), u_n \rangle_{X} \to \langle J(u), u \rangle_{X}.
    \]
    
    \item $J$ satisfies the \emph{$(S_{+})$-property} if for any sequence $\{u_n\} \subset X$ such that $u_n \rightharpoonup u$ in $X$ and
    \[
    \limsup_{n \to \infty} \langle J(u_n), u_n - u \rangle_{X} \le 0,
    \]
    we have $u_n \to u$ strongly in $X$.
\end{enumerate}
\end{definition}

\begin{theorem}\cite[Theorem 2.99]{carl}\label{the1}
Let $X$ be a reflexive Banach space, and let $\langle \cdot,\cdot \rangle_{X}$ denote the duality pairing between $X$ and its dual space $X^{*}$. If the operator $J:X\rightarrow X^{*}$ is bounded, coercive and pseudomonotone  then there exists a solution to the equation $J(u)=b$ for any $b\in X^{*}.$
\end{theorem}
\begin{lemma}\label{l1}
 Let $\{u_n\}\subset W_{0}^{1,1}(\Omega)$ and $u\in W_{0}^{1,1}(\Omega)$. Moreover, $1<p^-\leq p_{i,n}\leq p^+<\infty$ and $p_{i,n}\rightarrow p_i$ a.e. in $\Omega$, $$\sum_{i=1}^{N} \int_{\Omega}\left|\frac{\partial u_n}{\partial  x_i} \right|^{p_{i,n}(x)}\dx<\infty \quad \text{and} \quad \frac{\partial u_n}{\partial  x_i}\rightharpoonup\frac{\partial u}{\partial  x_i}$$ in $L^1(\Omega).$ Then $\nabla u\in (L^{p_i}(\Omega))^{N}$ and $$\int_{\Omega}\left|\frac{\partial u}{\partial  x_i} \right|^{p_{i}(x)}\dx\leq \liminf_{n\rightarrow\infty}\sum_{i=1}^{N}\int_{\Omega}\left|\frac{\partial u_n}{\partial  x_i} \right|^{p_{i,n}(x)}\dx.$$
\end{lemma}
\begin{proof}
Let \( b \in L^\infty(\Omega) \). By a standard Young-type inequality, for each \( i = 1, \dots, N \), we have
\begin{align*}
\frac{\partial u_n}{\partial x_i} \cdot b 
\leq  \left| \frac{\partial u_n}{\partial x_i} \right|^{p_{i,n}(x)} 
+ \frac{1}{p'_{i,n}(x)} \left( \frac{|b|}{(p_{i,n}(x))^{1/p_{i,n}(x)}} \right)^{p'_{i,n}(x)},
\end{align*}
where \( p'_{i,n}(x) = \frac{p_{i,n}(x)}{p_{i,n}(x) - 1} \) is the conjugate exponent of \( p_{i,n}(x) \). Integrating over \( \Omega \) and summing over \( i \), we obtain
\begin{align*}
\sum_{i=1}^N \int_\Omega 
\left( \frac{\partial u_n}{\partial x_i} \cdot b - \frac{|b|^{p'_{i,n}(x)}}{p'_{i,n}(x)(p_{i,n}(x))^{p'_{i,n}(x)/p_{i,n}(x)}} \right) \dx
\leq \sum_{i=1}^N \int_\Omega \left| \frac{\partial u_n}{\partial x_i} \right|^{p_{i,n}(x)} \dx.
\end{align*}
Since \( \frac{\partial u_n}{\partial x_i} \rightharpoonup \frac{\partial u}{\partial x_i} \) in \( L^1(\Omega) \), and \( p_{i,n} \to p_i \) a.e. in \( \Omega \), we may pass to the limit using the Dominated Convergence Theorem to obtain
\begin{equation} \label{eq:limit-b}
\sum_{i=1}^N \int_\Omega 
\left( \frac{\partial u}{\partial x_i} \cdot b 
- \frac{|b|^{p'_i(x)}}{p'_i(x)(p_i(x))^{p'_i(x)/p_i(x)}} \right) 
\dx 
\leq \liminf_{n \to \infty} \sum_{i=1}^N 
\int_\Omega \left| \frac{\partial u_n}{\partial x_i} \right|^{p_{i,n}(x)}\dx:=L
\end{equation}
where \( p'_i(x) = \frac{p_i(x)}{p_i(x) - 1} \).
For each \( k > 0 \), consider the function
\[
b = p_i(x) \left| \frac{\partial u}{\partial x_i} \right|_k^{\frac{1}{p'_i(x)-1}} 
\cdot \frac{\partial u}{\partial x_i} 
\left/ \left| \frac{\partial u}{\partial x_i} \right| \right.,
\quad \text{where } \left| \frac{\partial u}{\partial x_i} \right|_k 
:= \min \left\{ \left| \frac{\partial u}{\partial x_i} \right|, k \right\}.
\]
Substituting this choice of \( b \) into \eqref{eq:limit-b} gives
\begin{align*}
\sum_{i=1}^N \int_\Omega 
\left( 
p_i(x) \left| \frac{\partial u}{\partial x_i} \right|_k 
\cdot \left| \frac{\partial u}{\partial x_i} \right|_k^{\frac{1}{p'_i(x) - 1}} 
- \frac{p_i(x)}{p'_i(x)} 
\cdot \left| \frac{\partial u}{\partial x_i} \right|_k^{\frac{p'_i(x)}{p'_i(x) - 1}} 
\right) \dx \leq L,
\end{align*}
which implies
\begin{equation} \label{eq:bounded-conv}
\sum_{i=1}^N \int_\Omega 
\left| \frac{\partial u}{\partial x_i} \right|_k^{p_i(x)}\dx
\leq L.
\end{equation}
Since \( \left| \frac{\partial u}{\partial x_i} \right|_k \to \left| \frac{\partial u}{\partial x_i} \right| \) pointwise as \( k \to \infty \), and the integrands are nonnegative, we may apply the Monotone Convergence Theorem to \eqref{eq:bounded-conv} to obtain
\[
\sum_{i=1}^N \int_\Omega 
\left| \frac{\partial u}{\partial x_i} \right|^{p_i(x)} \dx 
\leq L := \liminf_{n \to \infty} \sum_{i=1}^N 
\int_\Omega \left| \frac{\partial u_n}{\partial x_i} \right|^{p_{i,n}(x)} \dx.
\]
Thus, we conclude that \( \frac{\partial u}{\partial x_i} \in L^{p_i(x)}(\Omega) \) for each \( i \), and hence \( \nabla u \in (L^{p_i(x)}(\Omega))^N \), which completes the proof.
\end{proof}

\section{Elliptic Problem} \label{sec3}
In this section, we establish the existence of a weak solution to the elliptic problem~\eqref{1.1}. Subsection~\ref{subsec3.1} introduces a perturbed version of the problem, where the existence of a weak solution is proved for the perturbed problem using the theory of pseudomonotone operators. In Subsection~\ref{subsec3.2}, we pass to the limit to obtain the existence of a weak solution to the problem \eqref{1.1}.

Let $u:\Omega\rightarrow\R$ is a continuous function then define the space $W^{1,\vec{\textbf{p}}(u)}(\Omega)$ as 
$$W^{1,\vec{\textbf{p}}(u)}(\Omega)=\left\lbrace v\in L^{p_M(u)}(\Omega) :  \frac{\partial v}{\partial  x_i}\in L^{p_i(u)}(\Omega), \ \forall i=1,2,\ldots,N \right\rbrace $$ which is a norm space with the norm
$$\|v\|_{W^{1,\vec{\textbf{p}}(u)}(\Omega)}=\|v\|_{L^{p_M(u)}(\Omega)}+\sum_{i=1}^{N}\left\| \frac{\partial v}{\partial  x_i}\right\|_{ L^{p_i(u)}(\Omega)} $$ where
$$\left\|  \frac{\partial v}{\partial  x_i}\right\| _{L^{p_i(u)}(\Omega)}=\inf\left\lbrace \tau>0: \ \int_{\Omega}\left| \frac{\partial v}{\tau\partial  x_i} \right|^{p_i(u)}   \dx\leq 1  \right\rbrace \cdot$$ 
Next, consider the closure of $C_c^\infty(\Omega)$ with respect to the space  $W^{1,\vec{\textbf{p}}(u)}(\Omega)$ and denote it as $W^{1,\vec{\textbf{p}}(u)}_{0}(\Omega)$, i.e., $W^{1,\vec{\textbf{p}}(u)}_{0}(\Omega)=\overline{C_c^\infty(\Omega)}|^{\|\cdot\|_{W^{1,\vec{\textbf{p}}(u)}(\Omega)}}$ which is a norm space with the norm 
$$\|v\|_{W^{1,\vec{\textbf{p}}(u)}_{0}(\Omega)}=\sum_{i=1}^{N}\left\| \frac{\partial v}{\partial  x_i}\right\|_{ L^{p_i(u)}(\Omega)}\cdot $$
By employing Theorems \ref{denseness} and \ref{denseness_0} in conjunction with the fact that the set of all H\"older continuous functions is contained within the set of all log-H\"older continuous functions, we have the following result:
\begin{theorem}\label{denseness_1}
Let $\Omega\subset \R^N$ be a bounded domain with a Lipschitz boundary. Suppose that conditions $(p_{1})$-$(p_{2}) $ are satisfied and $u\in W_0^{1,\vec{\textbf{p}}(u)}(\Omega)$. Then $C_c^\infty(\Omega)$ is dense in $W_0^{1,\vec{\textbf{p}}(u)}(\Omega)$. 
\end{theorem}
\begin{proof}
Let $u\in W_0^{1,\vec{\textbf{p}}(u)}(\Omega)$. By using the condition $(p_1)$ and by Sobolev embedding theorem, we have
$$W^{1,p^-}_0(\Omega)\hookrightarrow C
^{0,1-\frac{N}{p^-}}(\overline{\Omega}).$$ Hence, there exists a constant $C>0$  such that
\begin{equation}\label{4.2}
  |u(x)-u(y)| \leq C \|u\|_{W^{1,p^-}_0(\Omega)} |x-y|^{\alpha}, 
  \quad \forall x,y \in \overline{\Omega}.
\end{equation}
Define the variable exponent $\vec{\textbf{q}}=(q_1,q_2,\ldots,q_N) \in (C_{+}(\overline{\Omega}))^{N}$ by
\[
q_i(x) := p_i(u(x)), \quad \forall x \in \Omega, \ i=1,2,\ldots,N.
\]
By hypothesis $(p_{2})$, for each $i=1,2,\ldots,N$ , we have
\begin{equation}\label{4.3}
 |q_{i}(x)-q_{i}(y)|= |p_i(u(x)) - p_i(u(y))| \leq c_{i}|u(x)-u(y)|, 
  \quad \forall x,y \in \overline{\Omega}.
\end{equation}
Combining \eqref{4.2} with \eqref{4.3}, we deduce that
\[
   |q_{i}(x)-q_{i}(y)|= |p_i(u(x)) - p_i(u(y))|
  \leq c_{i}C \|u\|_{W^{1,p^-}_0(\Omega)} |x-y|^{\alpha}, 
  \quad \forall x,y \in \overline{\Omega}.
\]
Thus, each $q_{i}$ is H\"older continuous.  
Since every H\"older continuous function is also log-H\"older continuous, there exists a constant $L > 0$ such that
\begin{equation*}
 |q_{i}(x)-q_{i}(y)| \leq \frac{-L}{\log|x-y|}, 
  \quad \forall x,y \in \overline{\Omega}, \ |x-y| < \tfrac{1}{2}.
\end{equation*}
Hence, $q_i$ is log-H\"older continuous for each $i=1,2,\ldots,N$.  
Applying Theorem~\ref{denseness}, we conclude that 
$\mathcal{C}_{c}^{\infty}(\Omega)$ is dense in $W_0^{1,\vec{\textbf{p}}(u)}(\Omega)$.
\end{proof}

We define a weak solution to the problem \eqref{1.1} as follows: 
\begin{definition}
 A function $u\in W^{1,\vec{\textbf{p}}(u)}_{0}(\Omega)$ is called a weak solution to problem \eqref{1.1} if
$$\sum_{i=1}^{N} \int_{\Omega}\left|\frac{\partial u}{\partial  x_i} \right|^{p_i(u)-2}\frac{\partial u}{\partial  x_i}\frac{\partial v}{\partial  x_i}\dx  =\int_{\Omega}f(x,u)v  \dx,
 \ \forall v\in W^{1,\vec{\textbf{p}}(u)}_{0}(\Omega).
$$
\end{definition}

\subsection{Perturbed problem}\label{subsec3.1}
We first consider the following auxiliary problem:
\begin{equation}\label{1.2}
\begin{cases}
-\Delta_{\vec{\mathbf{p}}(u)}u - \epsilon \Delta_{p^+}u = f(x,u), & \text{in } \Omega, \\
u = 0, & \text{on } \partial\Omega.
\end{cases}
\end{equation}

\begin{definition}
A function $u \in W^{1,p^+}_0(\Omega)$ is called a weak solution of \eqref{1.2} if
$$\sum_{i=1}^{N} \int_{\Omega}\left|\frac{\partial u}{\partial  x_i} \right|^{p_i(u)-2}\frac{\partial u}{\partial  x_i}\frac{\partial v}{\partial  x_i}\dx+\epsilon\int_{\Omega}|\nabla u|^{p^{+}-2} \nabla u \nabla v \dx =\int_{\Omega}f(x,u)v \dx, \ \forall v\in W^{1,p^+}_{0}(\Omega).$$
\end{definition}
We define the operator $I: W^{1,p^+}_{0}(\Omega)\rightarrow( W^{1, p^+}_{0}(\Omega))^{*}$ as 
\begin{equation*}
\langle I(u),v\rangle_{W^{1,p^+}_{0}(\Omega)}=\sum_{i=1}^{N} \int_{\Omega}\left|\frac{\partial u}{\partial  x_i} \right|^{p_i(u)-2}\frac{\partial u}{\partial  x_i}\frac{\partial v}{\partial  x_i}\dx+\epsilon\int_{\Omega}|\nabla u|^{p^{+}-2} \nabla u \nabla v \dx -\int_{\Omega} f(x,u) v\dx
\end{equation*}
where
$\langle \cdot,\cdot \rangle_{W^{1,p^+}_{0}(\Omega)}$ stands the duality map between $ W^{1,p^+}_{0}(\Omega)$ and its dual space $( W^{1,p^+}_{0}(\Omega))^{*},$ for the simplicity, we write it as $\langle \cdot,\cdot \rangle$. We see that $u\in  W^{1,p^+}_{0}(\Omega)$ is a weak solution to \eqref{1.2} if and only if
$\langle I(u),v\rangle=0$
for all $ v\in  W^{1,p^+}_{0}(\Omega).$
\begin{theorem}\label{per_theorem}
Assume that conditions $(f_{1})$ and $(p_{1})$ are satisfied. Then problem \eqref{1.2} admits a nontrivial weak solution.
\end{theorem}
 Next, we verify each of the conditions required by Theorem~\ref{the1} in a systematic manner. These verifications are presented through Lemmas~\ref{corecive}–\ref{pseudomonotone}, where we establish the coercivity, boundedness, $S^+$-type property and pseudomonotonicity of the operator $I$.
\begin{lemma}\label{corecive}
Assume that conditions $(f_{1})$ and $(p_{1})$ are satisfied. Then the operator $I$ is coercive and bounded.
\end{lemma}
\begin{proof}
By using Sobolev embedding theorem and assumption $(f)$, we estimate
\begin{align*}
\langle I(u),u\rangle&=\sum_{i=1}^{N} \int_{\Omega}\left|\frac{\partial u}{\partial  x_i} \right|^{p_i(u)}\dx+\epsilon\int_{\Omega}|\nabla u|^{p^{+}}\dx  -\int_{\Omega} f(x,u)u\dx\\
&\geq \epsilon\int_{\Omega}|\nabla u|^{p^{+}}\dx-C\int_{\Omega}(u+|u|^{r})\dx\\
&\geq \epsilon \|u\|^{p^+}_{W^{1,p^+}_{0}(\Omega)}-C\|u\|^{r}_{W^{1,p^+}_{0}(\Omega)}-C\|u\|_{W^{1,p^+}_{0}(\Omega)},
\end{align*}
 where $1\leq r<p^{+}$. Hence, $I$ is coercive.

By condition $(f)$, the Nemitsky operator $N_{f}:W^{1,p^+}_{0}(\Omega)\rightarrow L^{r'}(\Omega)$  is well defined. Moreover, let  $i^{*}:L^{r'}(\Omega) \rightarrow (W^{1,p^+}_{0}(\Omega))^{*}$ be the adjoint of the embedding $i:W^{1,p^+}_{0}(\Omega)\rightarrow L^{r}(\Omega)$. Then, the operator $I_{1}:=i^{*}\circ N_{f}:W^{1,p^+}_{0}(\Omega)\rightarrow(W^{1,p^+}_{0}(\Omega))^{*}$ is continuous and bounded.

Now, for any $v\in W^{1,p^+}_{0}(\Omega)$, we estimate
\begin{align*}
\left|\langle I(u)+I_{1}(u),v\rangle \right|&=\sum_{i=1}^{N} \int_{\Omega}\left|\frac{\partial u}{\partial  x_i} \right|^{p_i(u)-1}\left|\frac{\partial v}{\partial  x_i}\right|\dx+\epsilon\int_{\Omega}|\nabla u|^{p^{+}-1} \left|\nabla v \right|\dx\\
& \leq \sum_{i=1}^{N} \int_{\Omega}\left(\left|\frac{\partial u}{\partial x_{i}}\right|^{p^{+}-1} \left|\frac{\partial v}{\partial x_{i}}\right|+\left|\frac{\partial u}{\partial x_{i}}\right|^{p^{-}-1} \left|\frac{\partial v}{\partial x_{i}}\right|\right)\dx+\epsilon\int_{\Omega}|\nabla u|^{p^{+}-1} \left|\nabla v \right|\dx\\
& \leq \sum_{i=1}^{N} \left\|\frac{\partial u}{\partial x_{i}}\right\|^{p^{+}-1}_{L^{p^{+}}(\Omega)}\left\|\frac{\partial v}{\partial x_{i}}\right\|_{L^{p^{+}}(\Omega)}+\sum_{i=1}^{N} \left\|\frac{\partial u}{\partial x_{i}}\right\|^{p^{-}-1}_{L^{p^{-}}(\Omega)}\left\|\frac{\partial v}{\partial x_{i}}\right\|_{L^{p^{-}}(\Omega)}\\
&+\epsilon \|\nabla u\|^{p^{+}-1}_{L^{p^{+}}(\Omega)}\|\nabla v\|_{L^{p^{+}}(\Omega)}\\
&\leq \|v\|_{W^{1,p^+}_{0}(\Omega)} \left[\sum_{i=1}^{N} \left\|\frac{\partial u}{\partial x_{i}}\right\|^{p^{+}-1}_{L^{p^{+}}(\Omega)}+\sum_{i=1}^{N} \left\|\frac{\partial u}{\partial x_{i}}\right\|^{p^{-}-1}_{L^{p^{-}}(\Omega)}+\epsilon \|\nabla u\|^{p^{+}-1}_{L^{p^{+}}(\Omega)}\right].
\end{align*}
Therefore, we have
\begin{align*}
\|I(u)+I_{1}(u)\|_{W^{1,p^+}_{0}(\Omega)}&=\sup\limits_{\|v\|_{W^{1,p^+}_{0}(\Omega)}\leq 1}\{|\langle I(u)+I_{1}(u),v\rangle_{W^{1,p^+}_{0}(\Omega)}|\}\\
&\leq  \sum_{i=1}^{N} \left\|\frac{\partial u}{\partial x_{i}}\right\|^{p^{+}-1}_{L^{p^{+}}(\Omega)}+\sum_{i=1}^{N} \left\|\frac{\partial u}{\partial x_{i}}\right\|^{p^{-}-1}_{L^{p^{-}}(\Omega)}+\epsilon \|\nabla u\|^{p^{+}-1}_{L^{p^{+}}(\Omega)},
\end{align*}
which shows that $I$ maps bounded sets into bounded sets, and hence  $I$ is a bounded operator.
\end{proof}
\begin{lemma}\label{s_+}
Assume that the conditions  $(f_{1})$ and $(p_{1})$ are satisfied. Then the operator $I$ satisfies the $(S_{+})$-property that is, if  $\{u_{n}\}\subset W^{1,p^+}_{0}(\Omega)$ such that $u_{n}\rightharpoonup u$  in $W^{1,p^+}_{0}(\Omega)$ and
$$\limsup_{n\rightarrow\infty}\langle I(u_{n}),u_{n}-u\rangle\leq 0,$$
then $\{u_{n}\}$ strongly converges to $u$ in $W^{1,p^+}_{0}(\Omega)$.
\end{lemma}
\begin{proof}
Let $\{u_{n}\}\subset W^{1,p^+}_{0}(\Omega)$ such that $u_{n}\rightharpoonup u$  in $W^{1,p^+}_{0}(\Omega)$ and $$\limsup_{n\rightarrow\infty}\langle I(u_{n}),u_{n}-u\rangle\leq 0.$$ We aim to prove that $\{u_{n}\}$ converges to $u$ strongly in $W^{1,p^+}_{0}(\Omega)$. 

From the definition of the operator $I$, we write
\begin{align}\label{s+1.0}
 \left\langle I\left(u_{n}\right), u_{n}-u\right\rangle&=
\sum_{i=1}^{N} \int_{\Omega}\left|\frac{\partial u_{n}}{\partial x_{i}}\right|^{p_{i}\left(u_{n}\right)-2} \frac{\partial u_{n}}{\partial x_{i}} \frac{\partial\left(u_{n}-u\right)}{\partial x_{i}}\dx+\epsilon \int_{\Omega}\left|\nabla u_{n}\right|^{p^{+}-2}  \nabla u_{n} \nabla\left(u_{n}-u\right)\dx\notag\\
&-\int_{\Omega} f(x,u_{n})(u_{n}-u)\dx.
\end{align}
Claim (a):
$$
\sum_{i=1}^{N} \int_{\Omega}\left|\frac{\partial u_{n}}{\partial x_{i}}\right|^{p_{i}\left(u_{n}\right)-2} \frac{\partial u_{n}}{\partial x_{i}} \frac{\partial\left(u_{n}-u\right)}{\partial x_{i}}\dx \geq o_{n}(1).
$$
To establish this, we proceed as follows. Using monotonicity and from Lemma~\ref{inequalities}, we have
\begin{align}\label{s+1.1}
 \sum_{i=1}^{N} \int_{\Omega}&\left|\frac{\partial u_{n}}{\partial x_{i}}\right|^{p_{i}\left(u_{n}\right)-2} \frac{\partial u_{n}}{\partial x_{i}} \frac{\partial\left(u_{n}-u\right)}{\partial x_{i}} \dx \notag\\
& =\sum_{i=1}^{N} \int_{\Omega}\left[\left|\frac{\partial u_n}{\partial x_{i}}\right|^{p_{i}\left(u_{n}\right)-2} \frac{\partial u_n}{\partial x_{i}}-\left|\frac{\partial u}{\partial x_{i}}\right|^{p_{i}\left(u_{n}\right)-2} \frac{\partial u}{\partial x_{i}}\right]\frac{\partial\left(u_{n}-u\right)}{\partial x_{i}}\dx\notag  \\
& +\sum_{i=1}^{N} \int_{\Omega}\left|\frac{\partial u}{\partial x_{i}}\right|^{p_{i}\left(u_{n}\right)-2} \frac{\partial u}{\partial x_{i}} \frac{\partial\left(u_{n}-u\right)}{\partial x_{i}}\dx \notag\\
& \geq \sum_{i=1}^{N} \int_{\Omega}\left|\frac{\partial u_{n}}{\partial x_{i}}-\frac{\partial u}{\partial x_{i}}\right|^{p_{i}\left(u_{n}\right)-2}\dx + \sum_{i=1}^{N} \int_{\Omega}\left|\frac{\partial u}{\partial x_{i}}\right|^{p_{i}\left(u_{n}\right)-2} \frac{\partial u}{\partial x_{i}} \frac{\partial\left(u_{n}-u\right)}{\partial x_{i}}\dx\notag \\
& \geq \sum_{i=1}^{N} \int_{\Omega}\left|\frac{\partial u}{\partial x_{i}}\right|^{p_{i}\left(u_{n}\right)-2} \frac{\partial u}{\partial x_{i}} \frac{\partial\left(u_{n}-u\right)}{\partial x_{i}}\dx.  
\end{align}
Observe that
\begin{align*}
& \sum_{i=1}^{N} \int_{\Omega}\left| \left|\frac{\partial u}{\partial x_{i}}\right|^{p_{i}\left(u_{n}\right)-2} \frac{\partial u}{\partial x_{i}} - \left|\frac{\partial u}{\partial x_{i}}\right|^{p_{i}\left(u\right)-2} \frac{\partial u}{\partial x_{i}} \right| \frac{\partial\left(u_{n}-u\right)}{\partial x_{i}}\dx \notag \\
& \leq 2 \sum_{i=1}^{N} \int_{\Omega}\left(\left|\frac{\partial u}{\partial x_{i}}\right|^{p^{+}-1}+\left|\frac{\partial u}{\partial x_{i}}\right|^{p^{-}-1}\right)\frac{\partial\left(u_{n}-u\right)}{\partial x_{i}}\dx \notag\\
& =2 \sum_{i=1}^{N} \int_{\Omega}\left|\frac{\partial u}{\partial x_{i}}\right|^{p^{+}-1}\frac{\partial\left(u_{n}-u\right)}{\partial x_{i}}\dx+ 2 \sum_{i=1}^{N} \int_{\Omega}\left|\frac{\partial u}{\partial x_{i}}\right|^{p^{-}-1}\frac{\partial\left(u_{n}-u\right)}{\partial x_{i}}\dx\rightarrow 0,
\end{align*}
as $\{u_{n}\}\rightharpoonup u$  in $W^{1,p^+}_{0}(\Omega)$ and 
$\left|\frac{\partial u}{\partial x_{i}}\right|^{p^{+}-1} \in L^{ \frac{p^{+}}{p^{+}-1}}(\Omega).$

Similarly as above,
\begin{align}\label{s+1.3}
\sum_{i=1}^{N} \int_{\Omega}\left|\frac{\partial u}{\partial x_{i}}\right|^{p_{i}\left(u\right)-2} \frac{\partial u}{\partial x_{i}} \frac{\partial\left(u_{n}-u\right)}{\partial x_{i}}\dx \rightarrow 0 \text{ as } n\rightarrow \infty.
\end{align}
The validity of claim (a) follows directly from equations \eqref{s+1.1}-\eqref{s+1.3}.

In view of condition $(f)$, and by applying H\"older's inequality along with the Sobolev embedding theorem, we derive the following
\begin{align}\label{s+1.4}
\lim_{n \rightarrow \infty}\int_{\Omega} f(x,u_{n})(u_{n}-u)\dx=0.
\end{align}
Combining \eqref{s+1.0}, claim (a), and \eqref{s+1.4}, we deduce that
\begin{align*}
\epsilon \lim_{n \rightarrow \infty}\int_{\Omega}\left|\nabla u_{n}\right|^{p^{+}-2}  \nabla u_{n} \nabla\left(u_{n}-u\right) \dx\leq \lim_{n \rightarrow \infty}\left\langle I\left(u_{n}\right), u_{n}-u\right\rangle \\
 \leq \limsup_{n \rightarrow \infty}\left\langle I\left(u_{n}\right), u_{n}-u\right\rangle \leq 0
\end{align*}
which implies
\begin{align}\label{s+1.5}
\quad \lim_{n \rightarrow \infty} \int_{\Omega}\left|\nabla u_{n}\right|^{p^{+}-2} \nabla u_{n} \nabla\left(u_{n}-u\right) \dx\leq 0.
\end{align}
Consider,
\begin{align}\label{s+1.6}
 \int_{\Omega}\left|\nabla u_{n}\right|^{p^{+}-2} \nabla u_{n} \nabla\left(u_{n}-u\right)\dx&= \int_{\Omega}\left[\left|\nabla u_{n}\right|^{p^{+}-2} \nabla u_{n}-|\nabla u|^{p^{+}-2} \nabla u\right] \nabla \left(u_{n}-u\right)\dx \notag\\
&+\int_{\Omega}|\nabla u|^{p^{+}-2} \nabla u \nabla\left(u_{n}-u\right)\dx  \notag\\
& \geq C \int_{\Omega}\left|\nabla\left(u_{n}-u\right)\right|^{p^{+}}\dx+o_n(1).
\end{align}

Combining \eqref{s+1.5} and \eqref{s+1.6}, we conclude that
$$ 0\leq  \lim _{n \rightarrow \infty}\int_{\Omega}\left|\nabla\left(u_{n}-u\right)\right|^{p^{+}}\dx\leq 0$$
and hence,
$$ \lim _{n \rightarrow \infty} \int_{\Omega}\left|\nabla\left(u_{n}-u\right)\right|^{p^{+}}\dx=0.
$$
Therefore, $ u_{n} \rightarrow u$ strongly in $W_{0}^{1, p^{+}}(\Omega)$ thereby confirming that the operator satisfies the $(S_{+})$-property.
\end{proof}
\begin{lemma}\label{pseudomonotone}
Assume that the conditions  $(f_{1})$ and $(p_{1})$ are satisfied. Then the operator $I$ is the pseudomonotone, that is, if $\{u_{n}\}$ converges weakly to $u$ in $W_{0}^{1, p^{+}}(\Omega)$ and $$\limsup\limits_{n\rightarrow\infty}\langle I(u_{n}),u_{n}-u\rangle\leq 0,$$ imply  $I(u_n)\rightharpoonup I(u)$ and 
$\langle I(u_n), u_n\rangle \rightarrow \langle I(u),u\rangle$. \end{lemma}
\begin{proof}
Let $\{u_{n}\}\subset W^{1,p^+}_{0}(\Omega)$ be a sequence such that $u_{n}\rightharpoonup u$  in $W^{1,p^+}_{0}(\Omega)$ , and suppose that $$\limsup_{n\rightarrow\infty}\langle I(u_{n}),u_{n}-u\rangle\leq 0.$$ Then, by the $S^{+}$-type property (that is, Lemma \ref{s_+}), it follows that $u_{n} \rightarrow u$ strongly in $W_{0}^{1, p^{+}}(\Omega)$.

Let $v \in W_{0}^{1, p^{+}}(\Omega)$ be arbitrary. We decompose
\begin{align*}
\langle I(u_{n})-I(u), v\rangle=  \sum_{i=1}^{N} \int_{\Omega}\left(\left|\frac{\partial u_{n}}{\partial x_{i}}\right|^{p_{i}\left(u_{n}\right)-2} \frac{\partial u_{n}}{\partial x_{i}}-\left|\frac{\partial u}{\partial x_{i}}\right|^{p_{i}(u)-2} \frac{\partial u}{\partial x_{i}}\right) \frac{\partial v}{\partial x_{i}}\dx \\
 +\epsilon \int_{\Omega}\left(\left|\nabla u_{n}\right|^{p^{+}-2} \nabla u_{n}-|\nabla u|^{p^{+} - 2} \nabla u\right) \nabla v\dx-\int_{\Omega} (f(x,u_{n})-f(x,u)) v\dx.
\end{align*}
We verify the convergence of each term separately.

Claim (a):
$$\sum_{i=1}^{N} \int_{\Omega}\left(\left|\frac{\partial u_{n}}{\partial x_{i}}\right|^{p_{i}\left(u_{n}\right)-2} \frac{\partial u_{n}}{\partial x_{i}}-\left|\frac{\partial u}{\partial x_{i}}\right|^{p_{i}(u)-2} \frac{\partial u}{\partial x_{i}}\right) \frac{\partial v}{\partial x_{i}}\dx =o_n(1), \forall v \in W_{0}^{1, p^{+}}(\Omega).$$
Let $v \in W_{0}^{1, p^{+}}(\Omega)$, we write
\begin{align}\label{p_1.1}
 \sum_{i=1}^{N} \int_{\Omega}&\left(\left|\frac{\partial u_{n}}{\partial x_{i}}\right|^{p_{i}\left(u_{n}\right)-2} \frac{\partial u_{n}}{\partial x_{i}}-\left|\frac{\partial u}{\partial x_{i}}\right|^{p_{i}(u)-2} \frac{\partial u}{\partial x_{i}}\right) \frac{\partial v}{\partial x_{i}}\dx \notag\\
& =\sum_{i=1}^{N} \int_{\Omega}\left(\left|\frac{\partial u_{n}}{\partial x_{i}}\right|^{p_{i}\left(u_{n}\right)-2} \frac{\partial u_{n}}{\partial x_{i}}-\left|\frac{\partial u}{\partial x_{i}}\right|^{p_{i}(u_n)-2} \frac{\partial u}{\partial x_{i}}\right) \frac{\partial v}{\partial x_{i}}\dx\notag\\
& +\sum_{i=1}^{N} \int_{\Omega}\left(\left|\frac{\partial u}{\partial x_{i}}\right|^{p_{i}\left(u_{n}\right)-2} \frac{\partial u}{\partial x_{i}}-\left|\frac{\partial u}{\partial x_{i}}\right|^{p_{i}(u)-2} \frac{\partial u}{\partial x_{i}}\right) \frac{\partial v}{\partial x_{i}}\dx\notag\\
& := (I)+(II).
\end{align}
From H\"older's inequality and Lemma \ref{inequalities}, one gets
\begin{align*}
(I)&=\sum_{i=1}^{N} \int_{\Omega}\left|\left(\left|\frac{\partial u_{n}}{\partial x_{i}}\right|^{p_{i}\left(u_{n}\right)-2} \frac{\partial u_{n}}{\partial x_{i}}-\left|\frac{\partial u}{\partial x_{i}}\right|^{p_{i}(u_n)-2} \frac{\partial u}{\partial x_{i}}\right) \frac{\partial v}{\partial x_{i}}\right|\dx\\
&\leq \sum_{i=1}^{N}\left(\int_{\Omega}\left(\left|\frac{\partial u_{n}}{\partial x_{i}}\right|^{p_{i}\left(u_{n}\right)-2} \frac{\partial u}{\partial x_{i}}-\left|\frac{\partial u}{\partial x_{i}}\right|^{p_{i}(u_n)-2} \frac{\partial u}{\partial x_{i}}\right)^{\frac{p_i^{+}}{p_i^{+}-1}}\dx\right)^{\frac{p_i^{+}-1}{p_i^{+}}}\left\|\frac{\partial v}{\partial x_{i}}\right\|_{L^{p_i^{+}}(\Omega)}\\
& \leq C \sum_{i=1}^{N}\left(\int_{\left| \frac{\partial u_{n}}{\partial x_{i}}\right|+  \left|\frac{\partial u}{\partial x_{i}}\right|> 1}\left(\left| \frac{\partial u_{n}}{\partial x_{i}}\right|+  \left|\frac{\partial u}{\partial x_{i}}\right|\right)^{\frac{(p_i^{+}-2)p_i^{+}}{p_i^{+}-1}}\left| \frac{\partial u_{n}}{\partial x_{i}}- \frac{\partial u}{\partial x_{i}}\right|^{\frac{p_i^{+}}{p_i^{+}-1}}\dx\right)^{\frac{p_i^{+}-1}{p_i^{+}}}\left\|\frac{\partial v}{\partial x_{i}}\right\|_{L^{p_i^{+}}(\Omega)}\\
& + C \sum_{i=1}^{N}\left(\int_{\left| \frac{\partial u_{n}}{\partial x_{i}}\right|+  \left|\frac{\partial u}{\partial x_{i}}\right|\leq 1}\left| \frac{\partial u_{n}}{\partial x_{i}}- \frac{\partial u}{\partial x_{i}}\right|^{\frac{p_i^{+}}{p_i^{+}-1}}\dx\right)^{\frac{p_i^{+}-1}{p_i^{+}}}\left\|\frac{\partial v}{\partial x_{i}}\right\|_{L^{p_i^{+}}(\Omega)}\\
& \leq C \sum_{i=1}^{N}\left(\int_{\Omega}\left(\left| \frac{\partial u_{n}}{\partial x_{i}}\right|+  \left|\frac{\partial u}{\partial x_{i}}\right|\right)^{p_i^{+}}\dx\right)^{\frac{p_i^{+}-2}{p_i^{+}}}\left\|\frac{\partial\left(u_{n}-u\right)}{\partial x_{i}}\right\|_{L^{p_i^{+}}(\Omega)}\left\|\frac{\partial v}{\partial x_{i}}\right\|_{L^{p_i^{+}}(\Omega)}\dx\\
& +C\sum_{i=1}^{N}\left\|\frac{\partial\left(u_{n}-u\right)}{\partial x_{i}}\right\|_{L^{^{\frac{p_i^{+}-1}{p_i^{+}}}}(\Omega)}\left\|\frac{\partial v}{\partial x_{i}}\right\|_{L^{p_i^{+}}(\Omega)}
\end{align*} 
since, $u_{n} \rightarrow u$ strongly in $W_{0}^{1, p^{+}}(\Omega)$, it follows that each term in the above sum tends to zero. Therefore, we conclude that
\begin{equation}\label{p1.2}
 (I)\rightarrow 0, \text{ as } n\rightarrow\infty.   
\end{equation}
By the H\"older's inequality, we have
\begin{align}\label{p1.3}
\sum_{i=1}^{N} &\left(\left|\frac{\partial u}{\partial x_{i}}\right|^{p_{i}\left(u_{n}\right)-2} \frac{\partial u}{\partial x_{i}}-\left|\frac{\partial u}{\partial x_{i}}\right|^{p_{i}(u)-2} \frac{\partial u}{\partial x_{i}}\right) \frac{\partial v}{\partial x_{i}}\notag\\
&\leq 2\sum_{i=1}^{N} \left(\left|\frac{\partial u}{\partial x_{i}}\right|^{p_{i}^{+}-1} \frac{\partial v}{\partial x_{i}}+\left|\frac{\partial u}{\partial x_{i}}\right|^{p_{i}^{-}-1} \frac{\partial v}{\partial x_{i}}\right)\in 
L^{1}(\Omega)
\end{align}
as $\frac{\partial u}{\partial x_{i}},\frac{\partial v}{\partial x_{i}}\in 
L^{p_{i}^{+}}(\Omega)$ and $\frac{\partial u}{\partial x_{i}},\frac{\partial v}{\partial x_{i}}\in 
L^{p_{i}^{-}}(\Omega)$.

By using Dominated Convergence Theorem, \eqref{p1.3} and the fact that $p(u_n) \rightarrow p(u)$, we have
\begin{equation}\label{p1.4}
 (II)\rightarrow 0, \text{ as } n\rightarrow\infty.   
\end{equation}
The proof of the claim (a) follow by \eqref{p_1.1}, \eqref{p1.2} and \eqref{p1.4}.

Claim (b):
$$
\int_{\Omega}\left(\left|\nabla u_{n}\right|^{p^{+}-2} \nabla u_{n}-|\nabla u|^{p^{+} - 2} \nabla u\right) \nabla v\dx=o_n(1).
$$
From H\"older's inequality and Lemma \ref{inequalities}, we obtain
\begin{align*}
& \int_{\Omega}\left(\left|\nabla u_{n}\right|^{p^{+}-2} \nabla u_{n}-|\nabla u|^{p^{+} - 2} \nabla u\right) \nabla v\dx \\
& \leq C\left(\int_{\Omega}\left(\left|\nabla u_{n}\right|^{p^{+}-2} \nabla u_{n}-|\nabla u|^{p^{+}-2} \nabla u\right)^{\frac{p^{+}}{p^{+}-1}}\dx\right)^{\frac{p^{+}-1}{p^{+}}}  \|\nabla v\|_{L^{p^{+}}(\Omega)}  \\
& \leq C\left(\int_{\Omega}(| \nabla u_{n}|+|\nabla u|) ^{\frac{p^{+}(p^{+}-2)}{p^{+}-1}}\left|\nabla u_{n}-\nabla u\right|^{\frac{p^{+}}{p^{+}-1}}\dx\right)^{\frac{p^{+}-1}{p^{+}}} \|\nabla v\|_{L^{p^{+}}(\Omega)} \\
& \leq C\left(\int_{\Omega}(| \nabla u_{n}|+|\nabla u|) ^{p^{+}}\dx\right)^{\frac{p^{+}-2}{p^{+}}} \|\nabla (u_{n}-u)\|_{L^{p^{+}}(\Omega)}\|\nabla v\|_{L^{p^{+}}(\Omega)} \\
& \rightarrow 0
\end{align*}
as $u_{n} \rightarrow u$ in $W_{0}^{1, p^{+}}(\Omega)$, which proves the claim (b).

Claim (c):
$$
\int_{\Omega} (f(x,u_{n})-f(x,u)) v\dx=o_n(1).
$$
By H\"older's inequality, we estimate
\begin{align}\label{p1.5}
 \int_{\Omega} f(x,u_{n})v\dx&\leq C\int_{\Omega}(v+|u_{n}|^{r-1}v)\dx\notag\\
 &\leq C\|v\|_{L^{1}(\Omega)}+C\|u_{n}\|^{r-1}_{L^{r}(\Omega)}\|v\|_{L^{r}(\Omega)}<\infty.
\end{align}
Since $f$ is continuous, and $u_{n}\rightarrow u$ almost everywhere in $\Omega$, it follows from the Dominated Convergence Theorem, together with estimate \eqref{p1.5}, that
$$
\lim _{n \rightarrow \infty}\int_{\Omega} f(x,u_{n}) v\dx=\int_{\Omega} f(x,u) v\dx,
$$
which establishes claim (c).

From claim (a), (b) and (c), we have
\begin{align}\label{p1.6}
\left\langle I\left(u_{n}\right), v\right\rangle \to \langle I(u), v\rangle, \ \forall \ v \in W_{0}^{1, p^{+}}(\Omega).
\end{align}
It remains to prove
$$
\left\langle I\left(u_{n}\right), u_{n}\right\rangle \rightarrow\langle I(u), u\rangle .
$$
From \eqref{p1.6}, we have
\begin{align*}
\left\langle I\left(u_{n}\right), u_{n}\right\rangle-\langle I(u), u\rangle&= \left\langle I\left(u_{n}\right), u_{n}\right\rangle-\left\langle I\left(u_{n}\right), u\right\rangle\\
&+
\left\langle I\left(u_{n}\right), u\right\rangle-\langle I(u), u\rangle \\
&= \left\langle I\left(u_{n}\right), u_{n}-u\right\rangle+\left\langle I\left(u_{n}\right)-I(u), u\right\rangle \\
&=  \left\langle I\left(u_{n}\right), u_{n}-u\right\rangle+o_{n}(1) .
\end{align*}
It is given that
$$
\limsup_{n \rightarrow \infty}\left\langle I(u_{n}), u_{n}-u\right\rangle \leq 0.
$$
 On the other hand, as established in Lemma~\ref{s_+}, we have 
 $$
\limsup_{n \rightarrow \infty}\left\langle I(u_{n}), u_{n}-u\right\rangle \geq 0.
$$ Combining these two inequalities, we deduce that
$$\lim_{n \rightarrow \infty}\left\langle I(u_{n}), u_{n}-u\right\rangle = 0$$
which completes the proof of the pseudo-monotonicity of the operator 
$I$.   
\end{proof}

 \textit{Proof of the Theorem \ref{per_theorem}}. By  Lemmas \ref{corecive}-\ref{pseudomonotone}, all the  conditions of Theorem \ref{the1}  are satisfied for the operator $I$. Hence, for a given $\epsilon>0$, there exists $u_{\epsilon}\in W^{1,p^+}_{0}(\Omega)$ such that
 $\langle I(u_{\epsilon}),v\rangle=0$ for all $v\in W^{1,p^+}_{0}(\Omega)$. Also, by using the fact that $f(\cdot,0)<0$ implies $u_{\epsilon}\neq 0$. \qed
 
\subsection{Passage to the Limit}\label{subsec3.2}
Take $\epsilon=\frac{1}{n}$ in \eqref{1.2}, by Theorem \ref{per_theorem} there exists $u_{n}\in W^{1,p^+}_{0}(\Omega)$ such that 
\begin{equation}\label{weak_sn_pr}
\sum_{i=1}^{N} \int_{\Omega}\left|\frac{\partial u_n}{\partial  x_i} \right|^{p_i(u_n)-2}\frac{\partial u_n}{\partial  x_i}\frac{\partial v}{\partial  x_i}\dx+\frac{1}{n}\int_{\Omega}|\nabla u_n|^{p^{+}-2} \nabla u_n \nabla v \dx=\int_{\Omega}f(x,u_n)v\dx,
\end{equation}
for all $ v\in W^{1,p^+}_{0}(\Omega)$.

Now, we are ready to prove the Theorem \ref{t1}. 

\textit{Proof of the Theorem \ref{t1}}. 
We divide the proof into several steps.
\begin{itemize}
\item \textbf{STEP 1:} $\{u_n\}$ is bounded in $W^{1,p^-}_{0}(\Omega)$.

Taking $v=u_n$ in \eqref{weak_sn_pr}, we have
\begin{equation}\label{weak_sn_pr1}
\sum_{i=1}^{N} \int_{\Omega}\left|\frac{\partial u_n}{\partial  x_i} \right|^{p_i(u_n)}\dx+\frac{1}{n}\int_{\Omega}|\nabla u_n|^{p^{+}} \dx =\int_{\Omega}f(x,u_n)u_n\dx. 
\end{equation}
Using assumption $(f)$ on the nonlinearity, along with the Sobolev embedding theorem and  H\"older's inequality, we estimate the right-hand side
\begin{align}\label{st1.1}
\int_{\Omega}f(x,u_n)u_n\dx &\leq C  \int_{\Omega} (|u_n|+|u_n|^{r})\dx\notag\\
& \leq C \int_{\Omega} (1+|u_n|^{r})\dx\notag\\
& \leq C |\Omega|+C\left(\int_{\Omega}|\nabla u_n|^{p^-}\dx\right)^{\frac{r}{p^-}},
\end{align}
since $r<p^-$.
Moreover, we estimate
\begin{align}\label{st1.2}
\left(\int_{\Omega}|\nabla u_n|^{p^-}\dx\right)^{\frac{1}{p^-}}&\leq C\sum_{i=1}^{N}\left\|\frac{\partial u_n}{\partial x_{i}}\right\|_{L^{p^{-}}(\Omega)}\notag\\
&\leq C\sum_{i=1}^{N}\left(\int_{\left|\frac{\partial u_n}{\partial x_{i}}\right|\geq 1} \left|\frac{\partial u_n}{\partial x_{i}}\right|^{p^-}\dx+\int_{\left|\frac{\partial u_n}{\partial x_{i}}\right|< 1} \left|\frac{\partial u_n}{\partial x_{i}}\right|^{p^-}\dx\right)^{\frac{1}{p^-}}\notag\\
 &\leq C\sum_{i=1}^{N}\left(\int_{\left|\frac{\partial u_n}{\partial x_{i}}\right|\geq 1} \left|\frac{\partial u_n}{\partial x_{i}}\right|^{p_i(u_n)}\dx+\int_{\left|\frac{\partial u_n}{\partial x_{i}}\right|< 1} \dx \right)^{\frac{1}{p^-}}\notag\\
 &\leq C\sum_{i=1}^{N}\left(\int_{\Omega} \left|\frac{\partial u_n}{\partial x_{i}}\right|^{p_i(u_n)}\dx\right)^{\frac{1}{p^-}}+C.
\end{align}
Substituting estimates~\eqref{st1.1} and~\eqref{st1.2} into~\eqref{weak_sn_pr1}, we obtain
$$\sum_{i=1}^{N} \int_{\Omega}\left|\frac{\partial u_n}{\partial  x_i} \right|^{p_i(u_n)}\dx\leq C+C \sum_{i=1}^{N}\left( \int_{\Omega} \left|\frac{\partial u_n}{\partial x_{i}}\right|^{p_i(u_n)}\dx\right)^{\frac{r}{p^-}}. $$
 Since $r<p^-$, a standard Young-type inequality yields
\begin{equation}\label{s1.3}
\sum_{i=1}^{N}\int_{\Omega}\left|\frac{\partial u_n}{\partial  x_i} \right|^{p_i(u_n)}\dx< C
\end{equation}
and from~\eqref{weak_sn_pr1}, we also get
\begin{equation}\label{s1.4}
\frac{1}{n}\int_{\Omega}|\nabla u_n|^{p^{+}}\dx  < C.
\end{equation}
Therefore, using~\eqref{st1.2} and~\eqref{s1.3}, we conclude that $\{u_n\}$ is bounded in $W^{1,p^-}_{0}(\Omega)$.
As \( W^{1,p^-}_0(\Omega) \) is reflexive,  there exists \( u \in W^{1,p^-}_0(\Omega) \) such that
\[
u_n \rightharpoonup u \quad \text{weakly in } W^{1,p^-}_0(\Omega), \quad u_n(x) \to u(x) \quad \text{a.e. in } \Omega,
\]
and consequently,
\[
p_i(u_n(x)) \to p_i(u(x)) \quad \text{a.e. in } \Omega, \quad \text{for all } i = 1, \dots, N.
\]
%As $W^{1,p^-}_{0}(\Omega)$ is a reflexive space, there exists $u\in W^{1,p^-}_{0}(\Omega)$ such that  up to a subsequence, we have $u_{n}\rightharpoonup u$ weakly in $W^{1,p^-}_{0}(\Omega)$, $u_{n}(x)\rightarrow u(x)$ a.e. in $\Omega$ and $p_i(u_{n}(x))\rightarrow p_i(u(x))$ for each $i=1,2,\ldots,N$.
\item \textbf{STEP 2:} $u\in W^{1,\vec{\textbf{p}}(u)}_{0}(\Omega)$, that is,  $\sum_{i=1}^{N}\left\| \frac{\partial u}{\partial  x_i}\right\|_{ L^{p_i(u)}(\Omega)}<\infty$.

Define $p_{i,n}(x)=p_{i}(u_n(x))$ for all $i=1,\cdots,N$. Applying Lemma~\ref{l1} with this choice and using \eqref{s1.3}, we obtain
$$\sum_{i=1}^{N}\int_{\Omega}\left|\frac{\partial u}{\partial  x_i} \right|^{p_{i}(u)}\dx\leq \liminf_{n\rightarrow\infty}\sum_{i=1}^{N}\int_{\Omega}\left|\frac{\partial u_n}{\partial  x_i} \right|^{p_{i}(u_n)}\dx<C$$
which implies that $u\in W^{1,\vec{\textbf{p}}(u)}_{0}(\Omega)$ by Lemma \ref{rel}. 

\item \textbf{STEP 3:} $u$ satisfies the following inequality:
\begin{equation}\label{st3.1}
\int_{\Omega}f(x,u)(u-v) \dx\geq\sum_{i=1}^{N} \int_{\Omega}\left|\frac{\partial v}{\partial  x_i} \right|^{p_i(u)-2}\frac{\partial v}{\partial  x_i}\frac{\partial (u-v)}{\partial  x_i}\dx,
\ \forall v\in  W^{1,\vec{\textbf{p}}(u)}_{0}(\Omega).
\end{equation}
Let $v\in C_{c}^{\infty}(\Omega)$. Using Lemma~\ref{inequalities}, we obtain the following inequality
\begin{align}\label{st2.1}
\sum_{i=1}^{N} \int_{\Omega}\left|\frac{\partial u_n}{\partial  x_i} \right|^{p_i(u_n)-2}&\frac{\partial u_n}{\partial  x_i}\frac{\partial (u_n-v)}{\partial  x_i}\dx+\frac{1}{n}\int_{\Omega}|\nabla u_n|^{p^{+}-2} \nabla u_n \nabla (u_n-v) \dx\notag\\
&\geq \sum_{i=1}^{N} \int_{\Omega}\left|\frac{\partial v}{\partial  x_i} \right|^{p_i(u_n)-2}\frac{\partial v}{\partial  x_i}\frac{\partial (u_n-v)}{\partial  x_i}\dx+\frac{1}{n}\int_{\Omega}|\nabla v|^{p^{+}-2} \nabla v\nabla (u_n-v)\dx.
\end{align}
Using the weak formulation \eqref{weak_sn_pr} and inequality \eqref{st2.1}, we arrive at
\begin{align}\label{st3.4}
\int_{\Omega}f(x,u_n)(u_n-v)\dx\geq &\sum_{i=1}^{N} \int_{\Omega}\left|\frac{\partial v}{\partial  x_i} \right|^{p_i(u_n)-2}\frac{\partial v}{\partial  x_i}\frac{\partial (u_n-v)}{\partial  x_i}\dx+\frac{1}{n}\int_{\Omega}|\nabla v|^{p^{+}-2} \nabla v\nabla (u_n-v)\dx\notag\\
& =I_1+I_2
\end{align}
where
\begin{align*}
I_1 &= \sum_{i=1}^{N} \int_{\Omega} \left| \frac{\partial v}{\partial x_i} \right|^{p_i(u_n)-2} \frac{\partial v}{\partial x_i} \frac{\partial (u_n - v)}{\partial x_i} \dx \\
I_2 &= \frac{1}{n} \int_{\Omega} |\nabla v|^{p^+-2} \nabla v \cdot \nabla (u_n - v) \dx.
\end{align*}
We decompose \( I_1 \) as
\begin{align*}
 I_1=\sum_{i=1}^{N} \int_{\Omega}&\left(\left|\frac{\partial v}{\partial  x_i} \right|^{p_i(u_n)-2}\frac{\partial v}{\partial  x_i}-\left|\frac{\partial v}{\partial  x_i} \right|^{p_i(u)-2}\frac{\partial v}{\partial  x_i}\right)\frac{\partial (u_n-v)}{\partial  x_i}\dx\\
 &+\sum_{i=1}^{N} \int_{\Omega}\left|\frac{\partial v}{\partial  x_i} \right|^{p_i(u)-2}\frac{\partial v}{\partial  x_i}\frac{\partial (u_n-v)}{\partial  x_i}\dx \\
:= &I_1^{'}+I_1^{''}.
\end{align*}
By H\"older's inequality and the Dominated Convergence Theorem (since \( v \in C_c^{\infty}(\Omega) \)), we obtain
\begin{align}\label{I_1^{'}}
 I_1^{'}&\leq \sum_{i=1}^{N} \left(\int_{\Omega}\left(\left|\frac{\partial v}{\partial  x_i} \right|^{p_i(u_n)-2}\frac{\partial v}{\partial  x_i}-\left|\frac{\partial v}{\partial  x_i} \right|^{p_i(u)-2}\frac{\partial v}{\partial  x_i}\right)^{\frac{p^{-}}{p^{-}-1}} \dx\right)^{\frac{p^{-}-1}{p^{-}}}\left\|\frac{\partial (u_n-v)}{\partial x_{i}}\right\|_{L^{p^{-}}(\Omega)}\notag \\
&\leq c \sum_{i=1}^{N} \left(\int_{\Omega}\left(\left|\frac{\partial v}{\partial  x_i} \right|^{p_i(u_n)-2}\frac{\partial v}{\partial  x_i}-\left|\frac{\partial v}{\partial  x_i} \right|^{p_i(u)-2}\frac{\partial v}{\partial  x_i}\right)^{\frac{p^{-}}{p^{-}-1}}\dx \right)^{\frac{p^{-}-1}{p^{-}}}\notag \\
& \rightarrow 0, \text{ as } n \rightarrow \infty.
\end{align}
Similarly,
\begin{equation}\label{I_1^{''}}
I_1^{''}\rightarrow \sum_{i=1}^{N} \int_{\Omega}\left|\frac{\partial v}{\partial  x_i} \right|^{p_i(u)-2}\frac{\partial v}{\partial  x_i}\frac{\partial (u-v)}{\partial  x_i}\dx.
\end{equation}
By using \eqref{I_1^{'}} and \eqref{I_1^{''}}, we have
\begin{equation}\label{I_1}
I_1  \rightarrow \sum_{i=1}^{N} \int_{\Omega}\left|\frac{\partial v}{\partial  x_i} \right|^{p_i(u)-2}\frac{\partial v}{\partial  x_i}\frac{\partial (u-v)}{\partial  x_i}\dx, \text{ as } n \rightarrow \infty.  
\end{equation}
For \( I_2 \), using \eqref{s1.4} and H\"older's inequality
\begin{align*}
|I_2|&\leq\frac{1}{n}\int_{\Omega}|\nabla v|^{p^{+}-1} |\nabla (u_n-v)|\dx\\
& \leq\frac{1}{n} \left(\int_{\Omega}|\nabla v|^{p^{+}}\dx\right)^{\frac{p^{+}-1}{p^{+}}}\left(\int_{\Omega}|\nabla (u_n-v)|^{p^{+}}\dx\right)^{\frac{1}{p^{+}}}\\
& \leq\frac{1}{n} \left(\int_{\Omega}|\nabla v|^{p^{+}}\dx\right)+\frac{C}{n}\left(\int_{\Omega}|\nabla (u_n)|^{p^{+}}\dx\right)^{\frac{1}{p^{+}}}\\
& \leq\frac{1}{n} C+\left(\frac{1}{n}\right)^{\frac{p^{+}-1}{p^{+}}}C
\end{align*}
which implies 
\begin{equation}\label{I_2}
I_2  \rightarrow 0, \text{ as } n \rightarrow \infty.  
\end{equation}
Next, consider the decomposition
$$\int_{\Omega}f(x,u_n)(u_n-v)\dx=\int_{\Omega}f(x,u_n)(u_n-u)\dx+\int_{\Omega}f(x,u_n)(u-v)\dx.$$
In view of condition $(f)$, and by applying H\"older's inequality along with the Sobolev embedding theorem, we derive the following
$$
\lim_{n \rightarrow \infty}\int_{\Omega} f(x,u_{n})(u_{n}-u)\dx=0.
$$
Additionally, as established in Claim (c) of Lemma~\ref{pseudomonotone}, we have
$$
\lim _{n \rightarrow \infty}\int_{\Omega} f(x,u_{n}) (u-v)\dx=\int_{\Omega} f(x,u) (u-v)\dx.
$$
Therefore, we have
\begin{equation}\label{st3.5}
\lim_{n \rightarrow \infty}\int_{\Omega} f(x,u_{n})(u_{n}-v)\dx=\int_{\Omega} f(x,u)(u-v)\dx.
\end{equation}
After taking limit in \eqref{st3.4}, using \eqref{I_1}, \eqref{I_2} and \eqref{st3.5}, we have
\begin{equation*}
\int_{\Omega}f(x,u)(u-v) \dx\geq\sum_{i=1}^{N} \int_{\Omega}\left|\frac{\partial v}{\partial  x_i} \right|^{p_i(u)-2}\frac{\partial v}{\partial  x_i}\frac{\partial (u-v)}{\partial  x_i}\dx,
\ \forall v\in  C_{c}^{\infty}(\Omega).
\end{equation*}
By the density of \( C_c^{\infty}(\Omega) \) in \( W^{1,\vec{\textbf{p}}(u)}_{0}(\Omega) \) (Theorem~\ref{denseness_1}), inequality \eqref{st3.1} holds for all \( v \in W^{1,\vec{\textbf{p}}(u)}_{0}(\Omega) \).
\item \textbf{STEP 4:} Conclusion

Let \( z \in W^{1,\vec{\textbf{p}}(u)}_{0}(\Omega) \) and \( \delta > 0 \). Substituting \( v = u \mp \delta z \) into inequality \eqref{st3.1}, we obtain
\begin{equation*}
\pm\delta\int_{\Omega}f(x,u)z \dx\geq\pm\delta\sum_{i=1}^{N} \int_{\Omega}\left|\frac{\partial u}{\partial  x_i} \mp\delta\frac{\partial z}{\partial  x_i}\right|^{p_i(u)-2}\left(\frac{\partial u}{\partial  x_i} \mp\delta\frac{\partial z}{\partial  x_i}\right)\frac{\partial z}{\partial  x_i}\dx.
\end{equation*}
Dividing by \( \delta \) and letting \( \delta \to 0 \), we conclude
$$
\int_{\Omega}f(x,u)z\dx=\sum_{i=1}^{N} \int_{\Omega}\left|\frac{\partial u}{\partial  x_i} \right|^{p_i(u)-2}\frac{\partial u}{\partial  x_i}\frac{\partial z}{\partial  x_i}\dx.
$$
Since \( z \in W^{1,\vec{\textbf{p}}(u)}_0(\Omega) \) is arbitrary, this implies that \( u \) satisfies the weak formulation of problem~\eqref{1.1}. Therefore, \( u \) is a weak solution.
\end{itemize}\qed
\section{Parabolic Problem}\label{sec4}
This Section focuses on the parabolic problem \eqref{p1.1} and prove the existence of weak solution by employing a combination of time discretization, approximation arguments, and Schauder’s fixed point theorem.

As we have already discussed, for a given~$u$, the exponent $\vec{\textbf{p}}(b(u))$ is a vector in~$\mathbb{R}^N$. Hence, the suitable space for analyzing equation~\eqref{p1.1} is the anisotropic Sobolev space.
Let $u:\Omega\rightarrow\R$ is a continuous function then define the space $E_u$ as 
\begin{equation*}
E_u = \left\{
\begin{aligned}
v \in L^{\infty}(0,T;L^{2}(\Omega)) \ :\ & \int_{\Omega_T} \left| \frac{\partial v}{\partial x_i} \right|^{p_i(b(u))} \, dx < \infty, \quad \forall i = 1,2,\ldots,N,\\ 
& v(\cdot,t) \in V_t(\Omega) \ \text{a.e. } t \in (0,T)
\end{aligned}
\right\}.
\end{equation*}
with 
$$V_t(\Omega)=\left\lbrace v\in L^{2}(\Omega)\cap W^{1,p^-}_{0}(\Omega): \int_{\Omega}\left| \frac{\partial v}{\partial  x_i} \right|^{p_i(b(u(\cdot,t)))}   \dx<\infty, \ \forall i=1,2,\ldots,N  \right\rbrace.$$
%which is a norm space with the norm
%$$\|v\|_{W^{1,\vec{\textbf{p}}(u)}(\Omega)}=\|v\|_{L^{p^+}(\Omega)}+\sum_{i=1}^{N}\left\| \frac{\partial v}{\partial  x_i}\right\|_{ L^{p_i(u)}(\Omega)}.$$
Next, consider the closure of $C_c^\infty(\Omega)$ with respect to the space  $W^{1,\vec{\textbf{p}}(b(u))}(\Omega)$ and denote it as $W^{1,\vec{\textbf{p}}(b(u))}_{0}(\Omega)$, i.e., $W^{1,\vec{\textbf{p}}(b(u))}_{0}(\Omega)=\overline{C_c^\infty(\Omega)}|^{\|\cdot\|_{W^{1,\vec{\textbf{p}}(b(u))}(\Omega)}}$ which is a norm space with the norm 
$$\|v\|_{W^{1,\vec{\textbf{p}}(b(u))}_{0}(\Omega)}=\sum_{i=1}^{N}\left\| \frac{\partial v}{\partial  x_i}\right\|_{ L^{p_i(b(u))}(\Omega)}\cdot $$
%Define the space
%$$W=W_0^{1,\vec{\textbf{p}}(b(u))}(\Omega)\cap L^{2}(\Omega).$$
We define a weak solution to \eqref{p1.1}. 
\begin{definition}\label{para_weak_sol_def}
 A function $u\in E_u\cap C([0,T];L^2(\Omega))$ said to be a weak solution of problem \eqref{p1.1} if the following identity holds
$$-\int_{\Omega}u_{0}(x)v(x,0)\dx-\int_{0}^{T}\int_{\Omega}u v_t\dxt+\sum_{i=1}^{N} \int_{0}^{T}\int_{\Omega}\left|\frac{\partial u}{\partial  x_i} \right|^{p_i(b(u))-2}\frac{\partial u}{\partial  x_i}\frac{\partial v}{\partial  x_i}\dxt  =\int_{0}^{T}\int_{\Omega}fv \dxt
$$
for all test functions $v\in C^{1}(\overline{\Omega}_T)$ satisfying $v(\cdot,T)=0$.
\end{definition}
\subsection{Time-discrete problem}
Let $N_0$ be a fixed positive integer, and define the time step size by $h = T/N_0$. We first consider the following time-discretized version of problem \eqref{p1.1}
\begin{equation}\label{elli_pro_k}
\begin{cases}
\dfrac{u_k - u_{k-1}}{h} -\Delta_{\vec{\textbf{p}}(b(u_k))} u_k = [f]_h\bigl((k-1)h\bigr), & x \in \Omega,\\
u_k|_{\partial \Omega} = 0, & k = 1,2,\dots,N_0,
\end{cases}
\end{equation}
where the Steklov average $[f]_h$ of $f$ is defined as
\[
[f]_h(x,t) = \frac{1}{h} \int_t^{t+h} f(x,\tau)\dtau.
\]

It is easy to verify that for each fixed $t$, the function $[f]_h(\cdot) \in L^{(p^{-})'}(\Omega).$
\begin{theorem}
Let the condition $(p_{1})$ be satisfied, and suppose that $f\in W^{-1,(p^{-})'}(\Omega)$. Let $b:W^{1,p^-}_{0}(\Omega)\rightarrow \R$ such that $b$ is continuous and bounded. Then, the problem \eqref{elli_pro_k} admits at least one non-trivial weak solution.
\end{theorem}
\begin{proof}
To establish the result, we structure the proof in two key steps.

\textbf{Step 1: Approximation} 

For the first time step $k=1$, the corresponding perturbed problem reads
\begin{equation}\label{apprx_pro}
\begin{cases}
\dfrac{u - u_{0}}{h} -\Delta_{\vec{\textbf{p}}(b(u))} u-\epsilon \Delta_{p^+} u = [f]_h(0), & x \in \Omega,\\[10pt]
u|_{\partial \Omega} = 0. 
\end{cases}
\end{equation}
Let $w\in L^2({\Omega})$ be fixed. For each $\epsilon>0$,  we consider the following problem where exponent is not dependent on the unknown function
\begin{equation}
\begin{cases}
\dfrac{u - u_{0}}{h} -\Delta_{\vec{\textbf{p}}(b(w))} u-\epsilon \Delta_{p^+} u = [f]_h(0), & x \in \Omega,\\[10pt]
u|_{\partial \Omega} = 0. 
\end{cases}
\label{apprx_pro_0}
\end{equation}
By employing the theory of monotone operators, for a given $w\in L^2({\Omega})$, there exists a unique solution, say, $u_w\in W^{1,p^+}_{0}(\Omega)$ to the problem \eqref{apprx_pro_0} such that
$$\int_{\Omega}\dfrac{u_w-u_0}{h}v\dx+\sum_{i=1}^{N} \int_{\Omega}\left|\frac{\partial u_w}{\partial  x_i} \right|^{p_i(b(w))-2}\frac{\partial u_w}{\partial  x_i}\frac{\partial v}{\partial  x_i}\dx+\epsilon\int_{\Omega}|\nabla u_w|^{p^{+}-2} \nabla u_w \nabla v \dx =\int_{\Omega}[f]_h(0)v\dx, $$
for all $ v\in W^{1,p^+}_{0}(\Omega).$
Choosing $v= u_w$, in the above equation yields the following energy estimate
\begin{align*}
\frac{1}{2h} \int_\Omega u_w^2 \dx
+ \sum_{i=1}^{N} \int_\Omega \left| \frac{\partial u_w}{\partial x_i} \right|^{p_i(b(u_w))} \dx
+ \epsilon \int_\Omega |\nabla u_w|^{p^+}\dx \leq \frac{1}{2h} \int_\Omega u_0^2 \, \dx + C\|\nabla u_w\|_{L^{p^+}(\Omega)}\dx.
\end{align*} 
Thus, the standard Young-type inequality yields the uniform bound
$$\|u_w\|_{L^{2}{(\Omega)}}+\|u_w\|_{W^{1,p^+}_{0}(\Omega)}\leq C_1,$$
where $C_1>0$ is a constant independent of the choice of $w$. 
Hence, we have
$$\|u_w\|_{L^{2}{(\Omega)}}\leq C_1.$$
Define the functional $J:L^2({\Omega})\rightarrow L^2({\Omega})$ such that 
$J(w)=u_w$ where $u_w$ is the weak solution of \eqref{apprx_pro_0}.

Claim: The mapping $J$ is a continuous.

Let $\{w_n\} \subset L^2({\Omega})$ be a sequence such that
\[
w_n \to w_0 \quad \text{strongly in } L^2({\Omega}) \text{ as } n \to \infty.
\]
Let \( u_n := u_{w_n} \in L^2(\Omega) \) denote the corresponding weak solution of \eqref{apprx_pro_0}. We want to show that
\[
u_n \to u_{w_0} \quad \text{strongly in } L^2(\Omega),
\]
where \( u_{w_0} \) is the solution of  \eqref{apprx_pro_0} associated to \( w_0 \).

From prior estimates, we have
\[
\|\nabla u_n\|_{L^{p^+}(\Omega)} \leq C,
\]
for some constant \( C > 0 \) independent of \( n \).
Therefore, there exists $u\in W_{0}^{1,p^{+}}(\Omega)$ such that
\begin{align*}
u_n &\rightharpoonup u \text{ in } W_{0}^{1,p^{+}}(\Omega),\\
u_n &\to u \text{ in } L^2(\Omega), \\
u_n(x) &\to u(x) \quad \text{a.e. in } \Omega.
\end{align*}
It remains to prove that, $u_{w_0} \equiv u$ i.e., $u$ is the weak solution corresponding to $w_0$.

Since each $u_n$ satisfies the weak formulation of \eqref{apprx_pro_0}, we have
\[
\int_\Omega \frac{u_n - u_0}{h} v \dx+
\sum_{i=1}^{N} \int_{\Omega}\left|\frac{\partial u_n}{\partial  x_i} \right|^{p_i(b(w_n))-2}\frac{\partial u_n}{\partial  x_i}\frac{\partial v}{\partial  x_i}\dx+ \epsilon \int_\Omega |\nabla u_n|^{p^{+}-2} \nabla u_n \cdot \nabla v \dx
= \int_\Omega [f]_h(0) v \dx,
\]
for all $v\in W_{0}^{1,p^{+}}(\Omega)$. Now, using the standard test function $v = u_n - \varphi$, and the monotonicity of the associated operators, we obtain
\begin{align}\label{pa1.1}
\int_\Omega \frac{(u_0 - \varphi) (u_n - \varphi)}{h} \dx
&+ \int_\Omega [f]_h(0) (u_n - \varphi)\dx
- \sum_{i=1}^{N} \int_{\Omega}\left|\frac{\partial \varphi}{\partial  x_i} \right|^{p_i(b(w_n))-2}\frac{\partial \varphi}{\partial  x_i}\frac{\partial (u_{n}-\varphi)}{\partial  x_i}\dx\notag\\
&- \epsilon \int_{\Omega} |\nabla \varphi|^{p^{+}- 2} \nabla \varphi \cdot \nabla (u_n - \varphi)\dx \geq 0
\quad \text{for all } \varphi \in W_0^{1, p^+}(\Omega).
\end{align}
By the continuity of $p_i$ and $b$, we have 
\begin{equation}\label{pa1.3}
p_i(b(w_n)) \to p_i(b(w_0)) \quad \text{a.e. in } \Omega.
\end{equation}
Moreover, using the integrability of the terms
\begin{equation}\label{pa1.4}
\int_{\Omega}  \left(\left|\frac{\partial \varphi}{\partial x_i}\right|^{p_i(b(w_n))-2} 
\frac{\partial \varphi}{\partial x_i}\right)^{\frac{p^{+}}{p^{+}-1}}\dx
\leq
\int_{\Omega} \left| \frac{\partial \varphi}{\partial x_i} \right|^{p^+}\dx<\infty
\quad \text{as } \varphi \in W_0^{1,p^+}(\Omega).
\end{equation}
By using the dominated convergence theorem, \eqref{pa1.3} and \eqref{pa1.4}, we get
\[
\sum_{i=1}^{N} \int_{\Omega}\left|\frac{\partial \varphi}{\partial  x_i} \right|^{p_i(b(w_n))-2}\frac{\partial \varphi}{\partial  x_i}\frac{\partial (u_{n}-\varphi)}{\partial  x_i}\dx\to \sum_{i=1}^{N} \int_{\Omega}\left|\frac{\partial \varphi}{\partial  x_i} \right|^{p_i(b(w_0))-2}\frac{\partial \varphi}{\partial  x_i}\frac{\partial (u-\varphi)}{\partial  x_i}\dx.
\]
By passing limit in \eqref{pa1.1}, we have
\begin{align*}
\int_{\Omega} \frac{u_0 - \varphi}{h} (u - \varphi) \dx
&+ \int_{\Omega} [f]_h(0) (u - \varphi) \dx
- \sum_{i=1}^{N} \int_{\Omega}\left|\frac{\partial \varphi}{\partial  x_i} \right|^{p_i(b(w_0))-2}\frac{\partial \varphi}{\partial  x_i}\frac{\partial (u-\varphi)}{\partial  x_i}\dx\\
&- \epsilon \int_{\Omega} |\nabla \varphi|^{p^{+}- 2} \nabla \varphi \cdot \nabla (u- \varphi)\dx \geq 0 \quad \text{for all } \varphi \in W_0^{1, p^+}(\Omega).
 \end{align*}
Take $\varphi = u \mp \delta z$ with $z \in W_0^{1,p^+}(\Omega)$ and $\delta > 0$, we obtain
\begin{align*}
\pm \left(  \int_{\Omega} \frac{u_0 - (u \mp \delta z)}{h} z\dx\right. 
&+ \int_{\Omega} [f]_h(0) z\dx
- \sum_{i=1}^{N} \int_{\Omega}\left|\frac{\partial (u \mp \delta z)}{\partial  x_i} \right|^{p_i(b(w_0))-2}\frac{\partial (u \mp \delta z)}{\partial  x_i}\frac{\partial z}{\partial  x_i}\dx
\\
&\left.- \epsilon \int_{\Omega} |\nabla(u \mp\delta z)|^{p^{+}- 2} \nabla(u \mp \delta z)  \nabla z \dx\right)
\geq 0.
\end{align*}
Taking $\delta \to 0$, in above we have
$$\int_{\Omega}\dfrac{u-u_0}{h}z\dx+\sum_{i=1}^{N} \int_{\Omega}\left|\frac{\partial u}{\partial  x_i} \right|^{p_i(b(w_0))-2}\frac{\partial u}{\partial  x_i}\frac{\partial z}{\partial  x_i}\dx+\epsilon\int_{\Omega}|\nabla u|^{p^{+}-2} \nabla u \nabla z\dx =\int_{\Omega}[f]_h(0)z\dx.$$
Since, $z\in W_0^{1, p^+}(\Omega)$ was arbitrary, we conclude that $u$ is indeed the weak solution of \eqref{apprx_pro_0} with input $w_0$.

Since, $u_{w_0}$ is the unique solution of \eqref{apprx_pro_0} when the data is $w_0 \in L^2(\Omega)$.
Thus from the uniqueness of solution, we have
\[
u_{w_0} \equiv u \quad \text{a.e. in } \Omega.
\]
Finally, we have
\[
w_n \to w_0 \implies u_n=J(w_n) \to J(w_0)=u_{w_0}, 
\]
establishing that $J$ is continuous.

Hence, by Schauder's fixed point theorem $J$ has a fixed point, i.e., for given  $\epsilon > 0$, there exists \( u_\epsilon \in L^2(\Omega) \) such that
\[
J(u_\epsilon ) = u_\epsilon.
\]
Therefore, $u_\epsilon$ solves the regularized problem \eqref{apprx_pro_0} with the data $u_\epsilon$. That is, for all $ v \in W_0^{1,p^+}(\Omega)$
\begin{align}\label{pa1.2}
\int_\Omega \frac{u_\epsilon - u_0}{h} v \dx
+ \sum_{i=1}^{N} \int_\Omega \left| \frac{\partial u_\epsilon}{\partial x_i} \right|^{p_i(b(u_\epsilon)) - 2}
\frac{\partial u_\epsilon}{\partial x_i} \frac{\partial v}{\partial x_i}\dx
+ \epsilon \int_\Omega |\nabla u_\epsilon|^{p^{+} - 2} \nabla u_\epsilon  \nabla v \dx
= \int_\Omega [f]_h(0) v\dx.
\end{align}
Hence, $ u_\epsilon\in W_0^{1,p^+}(\Omega)$ is the solution of the problem \eqref{apprx_pro}.

\textbf{Step 2: Passage to the limit} \( \epsilon \to 0 \).

Choosing $v= u_\epsilon$ as a test function in the variational formulation \eqref{pa1.2}, we obtain the following energy estimate
\begin{align}\label{pa1.7}
\frac{1}{2h} \int_\Omega u_\varepsilon^2 \dx
+ \sum_{i=1}^{N} \int_\Omega \left| \frac{\partial u_\epsilon}{\partial x_i} \right|^{p_i(b(u_\epsilon))} \dx
+ \epsilon \int_\Omega |\nabla u_\epsilon|^{p^+}\dx \leq \frac{1}{2h} \int_\Omega u_0^2 \dx + \int_\Omega [f]_h(0) u_\epsilon \dx.
\end{align}
To handle the right-hand side, we apply Young’s inequality for any $\delta>0$
\begin{equation}\label{young1}
\|[f]_h(0)\|_{L^{(p^+)'}(\Omega)}\|\nabla u_\epsilon\|_{L^{p^+}(\Omega)}\leq \delta \|\nabla u_\epsilon\|_{L^{p^+}(\Omega)}^{p^+}+\frac{1}{\delta^{\frac{1}{{p^+}-1}}}\|[f]_h(0)\|_{L^{(p^+)'}(\Omega)}^{(p^+)'}.
\end{equation}
Substituting \eqref{young1} into \eqref{pa1.7} and choosing $\delta$ small enough, we absorb the gradient term into the left-hand side. This yields the uniform bound
$$\sum_{i=1}^{N} \int_\Omega \left| \frac{\partial u_\epsilon}{\partial x_i} \right|^{p_i(b(u_\epsilon))} \dx+\frac{\epsilon}{2}\int_\Omega |\nabla u_\epsilon|^{p^+}\dx< C_2,$$
for some $C_2$ independent of $\epsilon$.
This in turn implies a uniform bound
$$\int_\Omega |\nabla u_\epsilon|^{p^-}\dx< C.$$
By reflexivity of \( W^{1,p^-}_0(\Omega) \), there exists a function \( u \in W^{1,p^-}_0(\Omega) \) such that, up to a subsequence,
\[
u_\epsilon \rightharpoonup u \quad \text{weakly in } W^{1,p^-}_0(\Omega),
\]
\[
u_\epsilon(x) \to u(x) \quad \text{a.e. in } \Omega,
\]
and, by continuity of the mappings \( p_i \circ b \), we also have
\[
p_i(b(u_\epsilon(x))) \to p_i(b(u(x))) \ \text{ a.e. in } \Omega,  \text{ for all } i = 1, \dots, N.
\]
Define \( p_{i,\epsilon}(x) := p_i(b(u_\epsilon(x))) \). Applying Lemma~\ref{l1} to the sequence \( \{u_\epsilon\} \) and using the above pointwise convergence of variable exponents, we obtain
\[
\sum_{i=1}^{N} \int_\Omega \left| \frac{\partial u}{\partial x_i} \right|^{p_i(b(u))} \dx
\leq \liminf_{\epsilon \to 0} \sum_{i=1}^{N} \int_\Omega \left| \frac{\partial u_\epsilon}{\partial x_i} \right|^{p_i(b(u_\epsilon))} \dx
\leq C.
\]
Proceeding as in Steps 3 and 4 of the proof of Theorem~\ref{t1}, we conclude that $u$  is the solution to \eqref{elli_pro_k} for $k=1$.

By repeating this procedure iteratively, we obtain that for a given $u_{k-1} \in W_0^{1,\vec{\textbf{p}}(b(u))}(\Omega),$ there exists  $u_k \in W_0^{1,\vec{\textbf{p}}(b(u))}(\Omega)$  such that
\begin{equation}\label{weak_k}
\int_\Omega \frac{u_k - u_{k-1}}{h} v \dx
+ \sum_{i=1}^{N} \int_\Omega \left| \frac{\partial u_k}{\partial x_i} \right|^{p_i(b(u_k)) - 2}
\frac{\partial u_k}{\partial x_i} \frac{\partial v}{\partial x_i} \dx
= \int_\Omega [f]_h((k-1)h) v \dx, \quad \forall v \in W_0^{1,\vec{\textbf{p}}(b(u))}(\Omega).
\end{equation}
\end{proof}
\subsection{Time-step size \texorpdfstring{$h\to 0$}{h→0}}
For each time step \( h = \frac{T}{N_0} \), we define the piecewise constant function (in variable $t$) $u_h$ as
\[
u_h(x,t) =
\begin{cases}
u_0(x), & t = 0, \\
u_1(x), & 0 < t \leq h, \\
\vdots & \\
u_j(x), & (j-1)h < t \leq jh, \\
\vdots & \\
u_{N_0}(x), & (N_0 - 1)h < t \leq N_0 h = T.
\end{cases}
\]
\textit{Proof of the Theorem \ref{t2}:}
Choosing $v= u_k$ in \eqref{weak_k}, we obtain the energy estimate
\begin{align}\label{pa1.5}
\frac{1}{2h} \int_\Omega u_{k}^2 \dx
+ \sum_{i=1}^{N} \int_\Omega \left| \frac{\partial u_k}{\partial x_i} \right|^{p_i(b(u_k))} \dx
 \leq \frac{1}{2h} \int_\Omega u_{k-1}^2 \dx + \|[f]_h((k-1)h)\|_{L^{(p^-)'}(\Omega)}\|\nabla u_k\|_{L^{p^-}(\Omega)}.
\end{align}
To relate the norms, we use the following estimate
\begin{align*}
\int_{\Omega}|\nabla u_k|^{p^-}\dx&\leq C\sum_{i=1}^{N}\left\|\frac{\partial u_k}{\partial x_{i}}\right\|_{L^{p^{-}}(\Omega)}^{p^{-}}\\
&\leq C\sum_{i=1}^{N}\left(\int_{\left|\frac{\partial u_k}{\partial x_{i}}\right|\geq 1} \left|\frac{\partial u_k}{\partial x_{i}}\right|^{p^-}\dx+\int_{\left|\frac{\partial u_k}{\partial x_{i}}\right|< 1} \left|\frac{\partial u_k}{\partial x_{i}}\right|^{p^-}\dx\right)\\
 &\leq C\sum_{i=1}^{N}\left(\int_{\left|\frac{\partial u_k}{\partial x_{i}}\right|\geq 1} \left|\frac{\partial u_k}{\partial x_{i}}\right|^{p_i(b(u_k))}\dx+\int_{\left|\frac{\partial u_k}{\partial x_{i}}\right|< 1} \dx \right)\\
 &\leq C\sum_{i=1}^{N}\int_{\Omega} \left|\frac{\partial u_k}{\partial x_{i}}\right|^{p_i(b(u_k))}\dx+C.
\end{align*}
By the Young's inequality, for a given $\delta>0$, we have
\begin{equation}\label{young}
\|[f]_h((k-1)h)\|_{L^{(p^-)'}(\Omega)}\|\nabla u_k\|_{L^{p^-}(\Omega)}\leq \delta \|\nabla u_k\|_{L^{p^-}(\Omega)}^{p^-}+\frac{1}{\delta^{\frac{1}{{p^-}-1}}}\|[f]_h((k-1)h)\|_{L^{(p^-)'}(\Omega)}^{(p^-)'}.
\end{equation}
Combining \eqref{pa1.5}-\eqref{young}, and absorbing terms, we obtain
$$\int_\Omega u_{k}^2 \dx
+ h\sum_{i=1}^{N} \int_\Omega \left| \frac{\partial u_k}{\partial x_i} \right|^{p_i(b(u_k))} \dx
 \leq  \int_\Omega u_{k-1}^2 \dx +Ch.$$

Summing over \( k = 1 \) to \( N_0 \), we deduce the uniform a priori estimate
\[
\int_\Omega u_h^2(x,t) \dx + \sum_{i=1}^{N}  \int_0^T \int_\Omega \left| \frac{\partial u_k}{\partial x_i} \right|^{p_i(b(u_k))} \dxt\leq \int_\Omega u_0^2 \dx + C T.
\]
Hence, we obtain the uniform bounds
\[
\|u_h\|_{L^\infty(0,T; L^2(\Omega))} + \sum_{i=1}^{N}\left\|\frac{\partial u_h}{\partial x_i}\right\|_{L^{p_i(b(u_h))}(\Omega_T)} + \|u_h\|_{L^{p^-}(0,T; W_0^{1,\vec{\textbf{p}}(b(u_h))}(\Omega))} \leq C.
\]

Thus, we can extract a subsequence such that
\[
\begin{aligned}
u_h & \rightharpoonup u \quad \text{weakly}-^* \text{ in }  L^\infty(0,T; L^2(\Omega)), \\
u_h & \rightharpoonup u \quad \text{weakly in } L^{p^{-}}(0,T; W_0^{1,p^{-}}(\Omega)), \\
\left|\frac{\partial u_h}{\partial x_i}\right|^{p_i(b(u_h)) - 2} \frac{\partial u_h}{\partial x_i} &\rightharpoonup \xi_i \quad \text{in } L^{(p^{-})'}(\Omega_T), \text{ for all } i=1,2,\ldots,N
\end{aligned}
\]
where, $\xi=(\xi_1,\xi_1,\cdots,\xi_N)\in (L^{(p^{-})'}(\Omega_T))^N$.
To prove \( u \) is a weak solution of problem \eqref{p1.1}, choose any \( \varphi \in C^1(\Omega_T) \) with \( \varphi(\cdot,T) = 0 \) and \( \varphi|_{\Gamma} = 0 \). Take \( \varphi(x,kh) \) as test function in \eqref{weak_k} to obtain for each \( k \)
\begin{align*}
\frac{1}{h} \int_\Omega u_k(x) \varphi(x, kh) \dx& - \frac{1}{h} \int_\Omega u_{k-1}(x) \varphi(x, kh) \dx+ \sum_{i=1}^{N} \int_\Omega \left| \frac{\partial u_k}{\partial x_i} \right|^{p_i(b(u_k)) - 2}
\frac{\partial u_k}{\partial x_i} \frac{\partial \varphi(x, kh)}{\partial x_i} \dx\\
&= \int_\Omega [f]_h((k-1)h) \varphi(x, kh) \dx.
\end{align*}
Summing over \( k = 1 \) to \( N_0 \), and using the boundary condition \( \varphi(\cdot, T) = 0 \), we derive
\begin{align*}
 h \sum_{k=1}^{N_0-1}& \int_\Omega u_h(x, kh) \frac{\varphi(x, kh) - \varphi(x, (k+1)h)}{h} \dx - \int_\Omega u_0(x) \varphi(x,h) \dx \\
& + h\sum_{k=1}^{N_0}\sum_{i=1}^{N} \int_\Omega \left| \frac{\partial u_h(x, kh)}{\partial x_i} \right|^{p_i(b(u_h(x, kh))) - 2}
\frac{\partial u_h(x, kh)}{\partial x_i} \frac{\partial \varphi(x, kh)}{\partial x_i} \dx  \\
&= h \sum_{k=1}^{N_0} \int_\Omega [f]_h((k-1)h) \varphi(x, kh) \dx.
\end{align*}
Furthermore, we observe
\begin{align*}
&h\sum_{k=1}^{N_0}\sum_{i=1}^{N} \int_\Omega \left| \frac{\partial u_h(x, kh)}{\partial x_i} \right|^{p_i(b(u_h(x, kh))) - 2}
\frac{\partial u_h(x, kh)}{\partial x_i} \frac{\partial \varphi(x, kh)}{\partial x_i} \dx\\
&=\sum_{i=1}^{N} \int_0^T\int_\Omega \left| \frac{\partial u_h(x, t)}{\partial x_i} \right|^{p_i(b(u_h(x, t))) - 2}
\frac{\partial u_h(x, t)}{\partial x_i} \frac{\partial \varphi(x, t)}{\partial x_i} \dxt\\
&+\sum_{k=1}^{N_0}\sum_{i=1}^{N}\int_{(k-1)h}^{kh} \int_\Omega \left| \frac{\partial u_h(x, t)}{\partial x_i} \right|^{p_i(b(u_h(x, t))) - 2}
\frac{\partial u_h(x, t)}{\partial x_i}\left( \frac{\partial \varphi(x, kh)}{\partial x_i} -\frac{\partial \varphi(x, t)}{\partial x_i}\right) \dxt\\
&\to \int_0^T\int_\Omega \xi\cdot \nabla \varphi \dxt \text{ as } h\to 0.
\end{align*}
By smoothness of \( \varphi \), as \( h \to 0 \), we recover
\[
- \int_0^T \int_\Omega u \frac{\partial \varphi}{\partial t} \dxt - \int_\Omega u_0(x) \varphi(x,0) \dx - \int_0^T \int_\Omega \xi \cdot \nabla \varphi \dxt = \int_0^T \int_\Omega f \varphi \dxt. 
\]
Following the method in \cite{chipot2019some}, we use monotonicity to conclude that \( \xi_i = \left|\frac{\partial u}{\partial  x_i} \right|^{p_i(b(u))-2}\frac{\partial u}{\partial  x_i} \) a.e. in \( \Omega_T \), for all $i=1,2,\ldots,N$.
Also, proceeding as in Lemma~\ref{l1}, it follows that 
\[
\frac{\partial u}{\partial x_i} \in L^{p_i(b(u))}(\Omega_T).
\]
 Choosing \( \varphi \in C^\infty_0(\Omega_T) \), we get
\[
- \int_0^T \int_\Omega u \frac{\partial \varphi}{\partial t} \dxt= \int_0^T \int_\Omega \xi \cdot \nabla \varphi \dxt + \int_0^T \int_\Omega f \varphi \dxt. 
\]
Hence, \(u_t \in X(\Omega_T)^{*}\). From the fact that \(u \in X(\Omega_T)\) and \(u_t \in X(\Omega_T)^{*}\), it follows from standard theory 
(see, for instance, \cite{diening2012monotone, roubicek2013nonlinear}) that 
\[
u \in C([0,T]; L^{2}(\Omega)).
\] This completes the proof. \qed
\bibliographystyle{siam}
\bibliography{mixed}
\end{document}